\documentclass[a4paper,reqno,10pt]{amsart}

\DeclareRobustCommand{\SkipTocEntry}[5]{}

\usepackage{hyperref}
\usepackage{a4wide}

\usepackage[justification=centering]{caption}
\usepackage[subrefformat=parens]{subcaption}

\usepackage{amssymb}
\usepackage{amstext}
\usepackage{amsmath}
\usepackage{amscd}
\usepackage{amsthm}
\usepackage{amsfonts}

\usepackage{graphicx}
\usepackage{latexsym}
\usepackage{mathrsfs}

\usepackage{enumitem}
\setlist[enumerate,1]{label={\upshape(\arabic*)}}
\setlist[enumerate,2]{label={\upshape(\alph*)}}

\usepackage{tikz}
\usetikzlibrary{cd,arrows,matrix,backgrounds,shapes,positioning,calc,decorations.markings,decorations.pathmorphing,decorations.pathreplacing}
\tikzset{blackv/.style={circle,fill=black,inner sep=3pt,outer sep=3pt},
         whitev/.style={circle,fill=white,draw=black,inner sep=3pt,outer sep=3pt},
         blabel/.style={
           circle, fill=white, draw=black, font=\scriptsize,
           inner sep=0.5pt, outer sep=0pt},
         redv/.style={circle,fill=red,inner sep=3pt,outer sep=3pt},
         block/.style={draw,rectangle split,rectangle split horizontal,rectangle split parts=#1},
         symbol/.style={
           draw=none,
           every to/.append style={
             edge node={node [sloped, allow upside down, auto=false]{$#1$}}}}
}

\usepackage{array}
\newcolumntype{C}{>{$}c<{$}}

\usepackage{makecell}

\newtheorem{theorem}{Theorem}[section]
\newtheorem{theoremi}{Theorem}
\newtheorem{corollaryi}[theoremi]{Corollary}

\newtheorem{corollary}[theorem]{Corollary}
\newtheorem{lemma}[theorem]{Lemma}
\newtheorem*{lemma*}{Lemma}
\newtheorem*{theorem*}{Theorem}
\newtheorem{proposition}[theorem]{Proposition}
\newtheorem{definition-proposition}[theorem]{Definition-Proposition}

\newtheorem{question}[theorem]{Question}

\theoremstyle{definition}
\newtheorem{definition}[theorem]{Definition}

\newtheorem{remark}[theorem]{Remark}
\newtheorem{example}[theorem]{Example}

\newcommand{\la}{\langle}
\newcommand{\ra}{\rangle}

\newcommand{\2}{\mathbf{2}}
\newcommand{\exk}{\ov{\kappa}}

\renewcommand{\AA}{\mathcal{A}}
\newcommand{\C}{\mathfrak{C}}
\newcommand{\CC}{\mathcal{C}}

\newcommand{\FF}{\mathcal{F}}
\newcommand{\FFF}{\mathsf{F}}
\newcommand{\HH}{\mathcal{H}}
\newcommand{\II}{\mathcal{I}}

\renewcommand{\SS}{\mathcal{S}}
\newcommand{\TT}{\mathcal{T}}
\newcommand{\TTT}{\mathsf{T}}
\newcommand{\UU}{\mathcal{U}}
\newcommand{\WW}{\mathcal{W}}

\newcommand{\WL}{\mathsf{W_L}}

\newcommand{\End}{\operatorname{End}\nolimits}
\newcommand{\op}{\operatorname{op}\nolimits}

\newcommand{\Image}{\operatorname{Im}\nolimits}
\newcommand{\Kernel}{\operatorname{Ker}\nolimits}
\newcommand{\Cokernel}{\operatorname{Coker}\nolimits}

\newcommand{\coker}{\Cokernel}
\newcommand{\im}{\Image}
\renewcommand{\ker}{\Kernel}
\newcommand{\un}{\underline}
\newcommand{\ov}{\overline}

\newcommand{\ot}{\leftarrow}

\newcommand{\clo}{\mathsf{CLO}}

\newcommand{\covd}{\lessdot}
\newcommand{\cov}{\gtrdot}

\DeclareMathOperator{\moduleCategory}{\mathsf{mod}} \renewcommand{\mod}{\moduleCategory}

\DeclareMathOperator{\simp}{\mathsf{sim}}

\DeclareMathOperator{\ice}{\mathsf{ice}}

\DeclareMathOperator{\theart}{\mathsf{tors-heart}}
\DeclareMathOperator{\itv}{\mathsf{itv}}

\DeclareMathOperator{\witv}{\mathsf{wide-itv}}
\DeclareMathOperator{\iitv}{\mathsf{ice-itv}}

\DeclareMathOperator{\wide}{\mathsf{wide}}

\DeclareMathOperator{\tors}{\mathsf{tors}}
\DeclareMathOperator{\torf}{\mathsf{torf}}

\DeclareMathOperator{\NC}{\mathsf{NC}}

\DeclareMathOperator{\brick}{\mathsf{brick}}
\DeclareMathOperator{\sbrick}{\mathsf{sbrick}}

\DeclareMathOperator{\jlabel}{\mathsf{j-label}}

\DeclareMathOperator{\Hasse}{\mathsf{Hasse}}

\DeclareMathOperator{\Sub}{\mathsf{Sub}}
\DeclareMathOperator{\Fac}{\mathsf{Fac}}
\DeclareMathOperator{\Filt}{\mathsf{Filt}}
\DeclareMathOperator{\add}{\mathsf{add}}
\DeclareMathOperator{\WR}{\mathsf{W_R}}

\DeclareMathOperator{\id}{\mathsf{id}}

\DeclareMathOperator{\jirr}{\mathsf{j-irr}^c}
\DeclareMathOperator{\mirr}{\mathsf{m-irr}^c}

\DeclareMathOperator{\CJR}{\mathsf{CJR}}
\DeclareMathOperator{\CMR}{\mathsf{CMR}}

\newcommand{\iso}{\cong}
\newcommand{\isoto}{\xrightarrow{\sim}}

\newcommand{\defl}{\twoheadrightarrow}

\newcommand{\sst}[1]{\substack{#1}}

\numberwithin{equation}{section}

\begin{document}
\title[From tors to wide and ICE]{From the lattice of torsion classes to the posets of wide subcategories and ICE-closed subcategories}

\author[H. Enomoto]{Haruhisa Enomoto}
\address{Graduate School of Science, Osaka Prefecture University, 1-1 Gakuen-cho, Naka-ku, Sakai, Osaka 599-8531, Japan}
\email{the35883@osakafu-u.ac.jp}

\begin{abstract}
  In this paper, we compute the posets of wide subcategories and ICE-closed subcategories from the lattice of torsion classes in an abelian length category in a purely lattice-theoretical way, by using the kappa map in a completely semidistributive lattice.
  As for the poset of wide subcategories, we give two more simple constructions via a bijection between wide subcategories and torsion classes with canonical join representations. More precisely, for a completely semidistributive lattice, we give two poset structures on the set of elements with canonical join representations: the kappa order (defined using the extended kappa map of Barnard--Todorov--Zhu), and the core label order (generalizing the shard intersection order for congruence-uniform lattices). Then we show that these posets for the lattice of torsion classes coincide and are isomorphic to the poset of wide subcategories. As a byproduct, we give a simple description of the shard intersection order on a finite Coxeter group using the extended kappa map.
\end{abstract}

\maketitle

\tableofcontents

\section{Introduction}

Let $\Lambda$ be a finite-dimensional algebra and $\mod\Lambda$ the category of finitely generated $\Lambda$-modules.
In representation theory of algebras, several classes of subcategories of $\mod\Lambda$ have been investigated.
In this paper, we mainly consider the following three classes: torsion classes, wide subcategories, and ICE-closed subcategories.
\emph{Torsion classes} have been playing an important role in the recent progress in representation theory via $\tau$-tilting theory \cite{AIR}.
\emph{Wide subcategories} are also classical and fundamental objects related to many things like torsion classes \cite{MS} and stability conditions.
\emph{ICE-closed subcategories} are subcategories closed under taking Images, Cokernels, and Extensions, which are introduced by the author in \cite{eno-rigid} as a common generalization of torsion classes and wide subcategories, and the relation to torsion classes are studied in \cite{ES}.

These classes of subcategories form posets under inclusion, and moreover, they are \emph{complete lattices}, that is, they have arbitrary joins and meets. Denote by $\tors\Lambda$, $\wide\Lambda$, and $\ice\Lambda$ the lattices of torsion classes, wide subcategories, and ICE-closed subcategories of $\mod\Lambda$ respectively. Among these lattices, the lattice property of $\tors\Lambda$ have recently been the focus of some attention (e.g. \cite{AP, BCZ, BTZ,DIRRT, thomas}).
The aim of this paper is to show that the lattice $\tors\Lambda$ remembers so much information about $\mod\Lambda$ that we can reconstruct $\ice\Lambda$ and $\wide\Lambda$ from it.

The relation between $\tors\Lambda$ and $\wide\Lambda$ is also of interest in combinatorics. Let $Q$ be a Dynkin quiver and $kQ$ its path algebra.Then $\tors kQ$ is isomorphic to the \emph{Cambrian lattice} and $\wide kQ$ is isomorphic to the \emph{non-crossing partition lattice} \cite{IT}. Similarly, let $\Pi_Q$ be the preprojective algebra of $Q$. Then $\tors\Pi_Q$ is isomorphic to the weak order of the Weyl group $W$ of $Q$ \cite{mizuno} and $\wide\Pi_Q$ is isomorphic to the \emph{shard intersection order on $W$} \cite{thomas-shard}, which is a relatively new poset structure on $W$ introduced by Reading \cite{reading-shard}. In both situations, it is not clear at first glance how $\tors \Lambda$ and $\wide \Lambda$ are related.

The main result of this paper is summarized as follows:
\begin{theoremi}\label{thm:intro-a}
  Let $\Lambda$ be a finite-dimensional algebra, and suppose that the lattice $L:= \tors\Lambda$ of torsion classes is given as an abstract lattice. Then we can compute the posets $\wide\Lambda$ and $\ice\Lambda$ only from the lattice $L$, without using any information on $\Lambda$ or $\mod\Lambda$.
\end{theoremi}
Actually, our results are valid for any abelian length category $\AA$. Denote by $\tors\AA$ the lattice of torsion classes in $\AA$.
To state our constructions in detail, we introduce some concepts in lattice theory. It is known that $\tors\AA$ is \emph{completely semidistributive} (Definition \ref{def:semidist}), hence we assume that a completely semidistributive lattice $L$ is given.
An element of $L$ is \emph{completely join-irreducible} if it cannot be written as a join of some elements non-trivially (Definition \ref{def:jirr}). We denote by $\jirr L$ the set of completely join-irreducible elements of $L$.
For each element $j \in \jirr L$, there is a unique element $j_*$ covered by $j$, and we define $\kappa(j) \in L$ as follows:
\[
  \kappa(j) = \max \{ x \in L \mid x \wedge j = j_*\}.
\]
Using this kappa map, for each interval $[a,b]$ in $L$, we define $\jlabel [a,b]$ as follows:
\[
  \jlabel [a,b] = \{ j \in \jirr L \mid j \leq b \text{ and } \kappa(j) \geq a \} \subseteq \jirr L.
\]
We remark that $\jlabel$ can be also described by using the \emph{join-irreducible labeling} of the Hasse quiver of $L$, see Theorem \ref{thm:jlabel-label}.

Now consider the case $L = \tors\AA$.
For an interval $[\UU,\TT]$ in $\tors\AA$, we define $\HH_{[\UU,\TT]}:= \TT \cap \UU^\perp$, which we call the \emph{heart} of $[\UU,\TT]$. We call a subcategory $\HH$ of $\AA$ arising in this way a \emph{torsion heart}, and denote by $\theart\AA$ the poset of torsion hearts.
It is shown in \cite{AP, ES} that $\wide\AA \subseteq \ice\AA\subseteq \theart\AA$ holds, and that there are lattice-theoretical characterizations of intervals in $\tors\AA$ whose hearts are wide subcategories and ICE-closed subcategories (Theorem \ref{thm:itv-char}).
Let us call such intervals in $\tors\AA$ \emph{wide intervals} and \emph{ICE intervals} respectively. Now we can state Theorem \ref{thm:intro-a} in detail.
\begin{theoremi}[= Theorem \ref{thm:main}]\label{thm:intro-b}
  Let $\AA$ be an abelian length category and $L := \tors\AA$. Then $\theart\AA$, $\wide\AA$, and $\ice\AA$ are isomorphic to the posets of subsets of $\jirr L$ of the form $\jlabel [a,b]$ for all, wide, and ICE intervals $[a,b]$ in $L$ respectively, ordered by inclusion.
\end{theoremi}

We roughly explain why this works. By \cite{DIRRT, BCZ}, there is a bijection between the set of \emph{bricks} in $\AA$ and $\jirr (\tors\AA)$ given by $B \mapsto \TTT(B)$, where $\TTT(B)$ is the smallest torsion class containing $B$.
Then we actually prove that $B \in \HH_{[\UU,\TT]}$ for a brick $B$ if and only if $\TTT(B) \in \jlabel[\UU,\TT]$, hence we can recover bricks contained in each torsion heart.
\begin{example}\label{ex:intro-alg}
  Let $k$ be a field and consider the algebra $\Lambda := k (1 \xleftarrow{b} 2 \xleftarrow{a} 3)/ \la ab \ra$. Then e.g. by using Geuenich's String Applet \cite{string-applet}, one obtain the lattice $L = \tors\Lambda$, which we show in Figure \ref{fig:intro-label}. Here we show the Hasse quiver of $L$, that is, we draw an arrow $x \to y$ if $x$ covers $y$.
  We can check $\jirr L = \{1,2,3,4,5\}$ and $\kappa(i) = \ov{i}$ for $i=1,2,3,4,5$.
  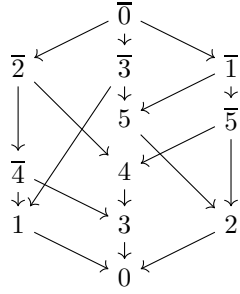
\begin{figure}[htp]
    \centering
    \begin{tikzpicture}
      [every node/.style={}, scale = 0.7]
      \begin{scope}
        \node (0) at (0, 0) {$0$};
        \node (1) at (-2, 1) {$1$};
        \node (3) at (0, 1) {$3$};
        \node (2) at (2, 1) {$2$};
        \node (4) at (0, 2) {$4$};
        \node (44) at (-2, 2) {$\ov{4}$};
        \node (5) at (0, 3) {$5$};
        \node (22) at (-2, 4) {$\ov{2}$};
        \node (55) at (2,3) {$\ov{5}$};
        \node (11) at (2,4) {$\ov{1}$};
        \node (33) at (0,4) {$\ov{3}$};
        \node (00) at (0,5) {$\ov{0}$};
      \end{scope}
  
      \begin{scope}
        \draw[->] (1) to (0);
        \draw[->] (2) to (0);
        \draw[->] (3) to (0);
        \draw[->] (4) to (3);
        \draw[->] (44) to (1);
        \draw[->] (44) to (3);
        \draw[->] (5) to  (2);
        \draw[->] (00) to (11);
        \draw[->] (00) to (22);
        \draw[->] (00) to (33);
        \draw[->] (55) to (2);
        \draw[->] (22) to (44);
        \draw[->] (22) to (4);
        \draw[->] (33) to (5);
        \draw[->] (33) to (1);
        \draw[->] (11) to (55);
        \draw[->] (11) to (5);
        \draw[->] (55) to (4);
      \end{scope}
    \end{tikzpicture}
    \caption{The Hasse quiver of $L := \tors\Lambda$}
    \label{fig:intro-label}
  \end{figure}
  \begin{table}[htp]
    \begin{tabular}{c|c}
      The poset of subcategories
      & The poset of subsets of $\jirr L$ \\ \hline \hline
      $\theart\Lambda$ &
      \makecell{
        $\{ \varnothing, 1, 2, 3, 4, 5, 13, 14, 15, 24, 25, 34, 35,$ \\
        $125,134,135,234,245,1245,2345,12345\}$
      } \\ \hline
      $\wide\Lambda$ &
      $\{\varnothing, 1, 2, 3, 4, 5, 13, 14, 35, 125, 234, 12345\}$
      \\ \hline
      $\ice\Lambda$ &
      \makecell{
        $\{ \varnothing, 1, 2, 3, 4, 5, 13, 14, 25, 34, 35,$ \\
        $125,134,234,2345,12345\}$
      }
    \end{tabular}
    \caption{Example of Theorem \ref{thm:intro-b}}
    \label{table:intro-ex}
  \end{table}

  For example, to compute $\jlabel[3, \ov{2}]$, we check which $j$ satisfies $j \leq \ov{2}$ and $\kappa(j) \geq 3$, and we obtain $\jlabel[3, \ov{2}] = \{1, 4\}$.
  Table \ref{table:intro-ex} shows how $\theart\Lambda$, $\wide\Lambda$, and $\ice\Lambda$ can be realized as posets of sets of join-irreducibles. In the second column, we write $125$ instead of $\{1, 2, 5 \}$ for example.
  Join-irreducibles $1,2,3,4,5$ correspond to bricks $S_1, S_2, S_3, P_3, P_2$ respectively, where $S_i$ and $P_i$ are simple and projective modules corresponding to each vertex $i$. Under this correspondence, the second column can be regarded as posets of bricks contained in each subcategory.
\end{example}

As for $\wide\Lambda$, we give two more simpler descriptions. It is known that there is an injection $\TTT \colon \wide\AA \to \tors\AA$ by taking the smallest torsion class containing each wide subcategory \cite{MS}, and we will recover $\wide\AA$ using this map as follows.

Let $L$ be a completely semidistributive lattice. A \emph{canonical join representation} of $x \in L$ is an expression $x = \bigvee A$ which is ``minimal'' in some sense (Definition \ref{def:cjr}), and denote by $L_0$ the set of elements of $L$ with canonical join representations. If $L$ is finite, then $L = L_0$ holds.
We note that our definition of canonical join representations is slightly different from the definition in \cite{BCZ, BTZ} when $L$ is infinite, see Remark \ref{rem:CJR-def}.

Then we give two poset structures on $L_0$, the \emph{kappa order} and the \emph{core label order}.
For $x \in L_0$ with a canonical join representation $x = \bigvee A$, Barnard--Todorov--Zhu introduced the \emph{extended kappa map $\exk(x)$} as follows, thereby obtaining a map $\exk \colon L_0 \to L$:
\[
  \exk(x) = \bigwedge \{\kappa(a) \mid a \in A\}.
\]
Using this map, we define the \emph{kappa order $\leq_\kappa$} on $L_0$ as follows:
\[
  x \leq_\kappa y : \Longleftrightarrow \text{$x \leq y$ and $\exk(x) \geq \exk(y)$}.
\]
On the other hand, we define the \emph{core label order $\leq_\clo$} on $L_0$ as follows:
\[
  x \leq_\clo y : \Longleftrightarrow \jlabel[x_\downarrow, x] \subseteq \jlabel[y_\downarrow, y],
\]
where $x_\downarrow := x \wedge \bigwedge \{ x' \in L \mid x' \covd x \}$.
This core label order is a generalization of that for finite congruence-uniform lattices (also known as the \emph{shard intersection order}), which were originally introduced for finite Coxeter groups by Reading \cite{reading-shard} and are studied in \cite{gm-flip,biclosed, muhle, reading-region}.

Now we can state our second main result of this paper.
\begin{theoremi}\label{thm:intro-c}
  Let $\AA$ be an abelian length category and put $L:=\tors\AA$. Then $\leq_\kappa$ and $\leq_\clo$ on $L_0$ coincide, and the map $\TTT \colon \wide\AA \to \tors\AA$ induces a poset isomorphism
  \[
    \wide\AA \isoto (L_0, \leq_\kappa) = (L_0, \leq_\clo).
  \]
\end{theoremi}

\begin{example}
  Consider the algebra $\Lambda$ and $L = \tors\Lambda$ in Example \ref{ex:intro-alg} again. Then the orbit of $\exk$ is given by $1 \mapsto \ov{1} \mapsto 2 \mapsto \ov{2}\mapsto 3 \mapsto \ov{3} \mapsto 4 \mapsto \ov{4} \mapsto 5 \mapsto \ov{5} \mapsto 1$ and $0 \mapsto \ov{0} \mapsto 0$. We can check that the kappa order and the shard intersection order on $L_0 = L$ coincide, and the Hasse diagram is given in Figure \ref{fig:intro-wide}.
  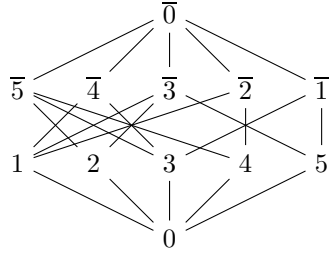
\begin{figure}[htp]
    \begin{tikzpicture}
      \node (0) at (0,0) {$0$};
      \node (1) at (-2,1) {$1$};
      \node (2) at (-1,1) {$2$};
      \node (3) at (0,1) {$3$};
      \node (4) at (1,1) {$4$};
      \node (5) at (2,1) {$5$};
      \node (11) at (2,2) {$\ov{1}$};
      \node (22) at (1,2) {$\ov{2}$};
      \node (33) at (0,2) {$\ov{3}$};
      \node (44) at (-1,2) {$\ov{4}$};
      \node (55) at (-2,2) {$\ov{5}$};
      \node (00) at (0,3) {$\ov{0}$};

      \draw (1) -- (0);
      \draw (2) -- (0);
      \draw (3) -- (0);
      \draw (4) -- (0);
      \draw (5) -- (0);
      \draw (11) -- (3);
      \draw (11) -- (5);
      \draw (22) -- (1);
      \draw (22) -- (4);
      \draw (33) -- (1);
      \draw (33) -- (2);
      \draw (33) -- (5);
      \draw (44) -- (1);
      \draw (44) -- (3);
      \draw (55) -- (2);
      \draw (55) -- (3);
      \draw (55) -- (4);
      \draw (00) -- (11);
      \draw (00) -- (22);
      \draw (00) -- (33);
      \draw (00) -- (44);
      \draw (00) -- (55);
    \end{tikzpicture}
    \caption{$\wide\Lambda \iso (L, \leq_\kappa) = (L, \leq_\clo)$}
    \label{fig:intro-wide}
  \end{figure}
\end{example}

As a byproduct, we obtain the following alternative description of Reading's original shard intersection on a finite Coxeter group order using the extended kappa map.
\begin{corollaryi}[= Proposition \ref{prop:coxeter-shard}]
  Let $W$ be a finite Coxeter group.
  Then the shard intersection order $\preceq$ \cite{reading-shard} on $W$ coincides with the kappa order $\leq_\kappa$ with respect to the right weak order $\leq$ on $W$, that is, $x \preceq y$ if and only if $x \leq y$ and $\exk(x) \geq \exk(y)$.
\end{corollaryi}
We remark that the kappa order and the core label order do not coincide in general even for finite congruence-uniform lattices, see Example \ref{ex:not-same}.

\addtocontents{toc}{\SkipTocEntry}
\subsection*{Computer program}
Since our results are purely combinatorial, one can do experiments in computer.
The author developed such a program \cite{program} on SageMath \cite{sage}, which computes various objects including $\wide\Lambda$, $\ice\Lambda$, and $\theart\Lambda$ if $\tors\Lambda$ is inputted (where we assume that $\tors\Lambda$ is finite).
For example, combining this program with Geuenich's String Applet \cite{string-applet}, one can compute the above things for any representation-finite special biserial algebra $\Lambda$.

\addtocontents{toc}{\SkipTocEntry}
\subsection*{Organization}
This paper is organized as follows.
In Section \ref{sec:2}, we collect basic results in lattice theory and representation theory of algebras which we use throughout this paper.
In Section \ref{sec:3}, we introduce the map $\jlabel$ and torsion hearts, then we prove Theorem \ref{thm:intro-b}.
In Section \ref{sec:4}, we first study the basics of canonical join representations, and discuss the relation between wide subcategories, bricks and canonical join representations of torsion classes. Then we introduce the kappa order and the core label order, and prove Theorem \ref{thm:intro-c}.

\addtocontents{toc}{\SkipTocEntry}
\subsection*{Conventions and notation}
Throughout this paper, \emph{we assume that all categories are skeletally small}, that is, the isomorphism classes of objects form a set. In addition, \emph{all subcategories are assumed to be full and closed under isomorphisms}. 
For an artinian ring $\Lambda$, we denote by $\mod\Lambda$ the category of finitely generated right $\Lambda$-modules.

\section{Preliminaries}\label{sec:2}
In this section, we give some background and tools on lattice theory and representation theory.
Although the material in this section is not new and can be found in e.g. \cite{BCZ, BTZ, DIRRT, RST,thomas} for the case of finite-dimensional algebras, we provide some short proofs in the setting of abelian length categories to make this paper self-contained and to use later.

First, we introduce some terminology.
Let $P$ be a poset. Then its \emph{Hasse quiver} $\Hasse P$ is the quiver defined as follows: The vertex set is $P$, and we draw an arrow $a \to b$ if $a > b$ and there is no $x$ in $P$ satisfying $a > x > b$. In this case, we say that \emph{$a$ covers $b$} and write $a \cov b$. We denote by $\Hasse_1 P$ the set of arrows in $\Hasse P$, and its element is called a \emph{Hasse arrow of $P$}.

A poset $L$ is called a \emph{complete lattice} if each subset $X \subseteq L$ has a least upper bound $\bigvee X$ and a greatest lower bound $\bigwedge X$. In particular, a complete lattice has the greatest element $\bigwedge \varnothing$ and the least element $\bigvee \varnothing$.

\subsection{The kappa map and the join-irreducible labeling}
In this subsection, we recall basics of completely semidistributive lattices which will be used throughout this paper.

\begin{definition}\label{def:jirr}
  Let $L$ be a complete lattice.
  \begin{enumerate}
    \item An element $j \in L$ is called \emph{completely join-irreducible} if $j = \bigvee X$ for some subset $X \subseteq L$ implies $j \in X$.
    \item Dually, an element $m \in L$ is called \emph{completely meet-irreducible} if $m = \bigwedge X$ for some subset $X \subseteq L$ implies $m \in X$.
  \end{enumerate}
  We denote by $\jirr L$ (resp. $\mirr L$) the set of completely join-irreducible (resp. completely meet-irreducible) elements of $L$.
\end{definition}
It is convenient to use the following notation when we consider join-irreducibles and meet-irreducibles.
\begin{definition}
  Let $L$ be a complete lattice and $a \in L$. We define $a_*$ and $a^*$ as follows.
  \begin{align*}
    a_* &:= \bigvee \{ x \in L \mid x < a \} \\
    a^* &:= \bigwedge \{ x \in L \mid x > a \}
  \end{align*}
\end{definition}
It is easily verified that $j \in L$ is completely join-irreducible if and only if $j \neq j_*$, and in this case, $j_*$ is a maximum element below $j$ and is the unique element covered by $j$. Dually, $m \in L$ is completely meet-irreducible if and only if $m^* \neq m$, and in this case, $m^*$ is a minimum element above $m$ and is the unique element which covers $m$.

Now let us recall \emph{completely semidistributive lattices}, which provide the framework of our study of the lattice of torsion classes.
\begin{definition}\label{def:semidist}
  A lattice $L$ is called \emph{completely semidistributive} if it is a complete lattice and satisfies the following conditions:
  \begin{enumerate}
    \item For $a, b \in L$ and $X \subseteq L$, if $a \vee x = b$ for every $x \in X$, then $a \vee (\bigwedge X) = b$ holds.
    \item For $a,b \in L$ and $X \subseteq L$, if $a \wedge x = b$ for every $x \in X$, then $a \wedge (\bigvee X) = b$ holds.
  \end{enumerate}
  In addition, if $L$ is a finite lattice, then we simply call $L$ a \emph{finite semidistributive lattice}.
\end{definition}
To each Hasse arrow of a completely semidistributive lattice, we can associate a completely join-irreducible element and a completely meet-irreducible element as follows. 
\begin{proposition}[{\cite[Proposition 9.1]{thomas}, \cite[Lemma 3.7]{RST}}]\label{prop:jirr-label}
  Let $L$ be a completely semidistributive lattice and $a \to b$ a Hasse arrow of $L$. Then the following hold.
  \begin{enumerate}
    \item $\{ x \in L \mid b \vee x = a \}$ has a minimum element, which is completely join-irreducible.
    \item $\{ x \in L \mid a \wedge x = b \}$ has a maximum element, which is completely meet-irreducible.
  \end{enumerate}
  Thus we obtain the following two maps $\gamma$ and $\mu$.
  \[\begin{tikzcd}[row sep = 0]
    \jirr L & \Hasse_1 L \lar["\gamma"'] \rar["\mu"] & \mirr L, \\
    \min \{ x \in L \mid b \vee x = a \}
    & (a \to b) \lar[mapsto] \rar[mapsto]
    & \max \{ x \in L \mid a \wedge x = b \}.
  \end{tikzcd}\]
\end{proposition}
\begin{proof}
  We only prove (1). Put $X := \{ x \in L \mid b \vee x = a \}$. Then $j := \bigwedge X$ belongs to $X$ by complete semidistributivity, and $j$ is clearly a minimum element of $X$.

  We will show that $j$ is completely join-irreducible. Suppose $j = \bigvee Y$ for some $Y \subseteq L$. Then for each $y \in Y$, we have $b \leq b \vee y \leq b \vee (\bigvee Y) = b \vee j = a$, thus either $b \vee y = b$ or $b \vee y = a$ holds. Since $b \vee (\bigvee Y) = a$ holds, there exists some $y \in Y$ satisfying $b \vee y = a$, thus $y \in X$. Since $j$ is the minimum element of $X$, it follows that $j \leq y$ holds. Thus $j \leq y \leq \bigvee Y = j$ holds, which implies $j = y \in Y$. 
\end{proof}
Therefore, we have the following two arrow labelings on $\Hasse L$.
\begin{definition}\label{def:label}
  Let $L$ be a completely semidistributive lattice. We define the \emph{join-irreducible labeling $\gamma \colon \Hasse_1 L \to \jirr L$} and the \emph{meet-irreducible labeling $\mu \colon \Hasse_1 L \to \mirr L$} as in Proposition \ref{prop:jirr-label}.
\end{definition}

\begin{example}
  Consider a lattice $L$ in Example \ref{ex:intro-alg}. Then Figure \ref{fig:label-ex} shows the join-irreducible labeling $\gamma$ on the Hasse quiver of $L$.
  \begin{figure}[htp]
    \centering
    \begin{tikzpicture}
      [every node/.style={inner sep=0.3pt}, scale = 0.8]
      \begin{scope}
        \node (0) at (0, 0) {$0$};
        \node (1) at (-2, 1) {$1$};
        \node (3) at (0, 1) {$3$};
        \node (2) at (2, 1) {$2$};
        \node (4) at (0, 2) {$4$};
        \node (44) at (-2, 2) {$\ov{4}$};
        \node (5) at (0, 3) {$5$};
        \node (22) at (-2, 4) {$\ov{2}$};
        \node (55) at (2,3) {$\ov{5}$};
        \node (11) at (2,4) {$\ov{1}$};
        \node (33) at (0,4) {$\ov{3}$};
        \node (00) at (0,5) {$\ov{0}$};
      \end{scope}
  
      \begin{scope}[every node/.style={blabel}]
        \draw[->] (1) to node {$1$} (0);
        \draw[->] (2) to node {$2$} (0);
        \draw[->] (3) to node {$3$} (0);
        \draw[->] (4) to node {$4$} (3);
        \draw[->] (44) to node {$3$} (1);
        \draw[->] (44) to node {$1$} (3);
        \draw[->] (5) to node[near end] {$5$} (2);
        \draw[->] (00) to node {$1$} (11);
        \draw[->] (00) to node {$2$} (22);
        \draw[->] (00) to node {$3$} (33);
        \draw[->] (55) to node {$3$} (2);
        \draw[->] (22) to node {$4$} (44);
        \draw[->] (22) to node[near start] {$1$} (4);
        \draw[->] (33) to node {$1$} (5);
        \draw[->] (33) to node {$2$} (1);
        \draw[->] (11) to node {$5$} (55);
        \draw[->] (11) to node {$3$} (5);
        \draw[->] (55) to node {$2$} (4);

      \end{scope}
    \end{tikzpicture}
    \caption{The join-irreducible labeling on $L$}
    \label{fig:label-ex}
  \end{figure}
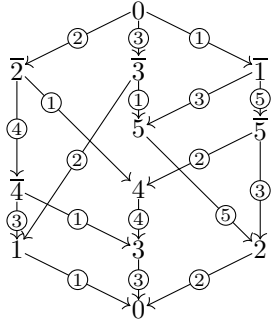  
\end{example}

The following characterization of these labelings is useful.
\begin{lemma}\label{lem:label-char}
  Let $L$ be a completely semidistributive lattice and $a \to b$ a Hasse arrow of $L$. Then the following hold.
  \begin{enumerate}
    \item An element $j$ of $L$ satisfies $\gamma(a \to b) = j$ if and only if $b \vee j = a$ and $b \wedge j = j_*$ hold.
    \item An element $m$ of $L$ satisfies $\mu(a \to b) = m$ if and only if $a \wedge m = b$ and $a \vee m = m^*$ hold.
  \end{enumerate}
  \[
    \begin{tikzpicture}[xscale = 0.8, yscale = -0.6]
      \node (mm) at (1,0) {$m^*$};
      \node (m) at (0,1) {$m$};
      \node (a) at (2,1) {$a$};
      \node (b) at (1,2) {$b$};
      \node (j) at (3,2) {$j$};
      \node (jj) at (2,3) {$j_*$};

      \draw[->] (mm) -- (m);
      \draw[->] (a) -- (b);
      \draw[->] (j) to (jj);
      \draw[dashed] (mm) -- (a) -- (j);
      \draw[dashed] (m) -- (b) -- (jj);
    \end{tikzpicture}
    \]
\end{lemma}
\begin{proof}
  We only prove (1), since (2) follows from duality.
  Suppose that $j$ satisfies $\gamma(a \to b) = j$. By the definition of $\gamma$, we have $b \vee j = a$. Since $j_* < j$, we have $b \vee j_* \neq a$ by the minimality of $j$, thus $b \leq b \vee j_* < b \vee j = a$, which implies $b \vee j_* = b$, or equivalently, $j_* \leq b$. Thus we have $j_* \leq b \wedge j \leq j$.
  If $b \wedge j = j$, then we have $j \leq b$, thus $b \vee j = b \neq a$ holds, which is a contradiction. Thus $b \wedge j \neq j$, hence $b \wedge j = j_*$.

  Conversely, suppose that $b \vee j = a$ and $b \wedge j = j_*$ hold. The first equality shows that $j$ belongs to the set $X := \{ x \in L \mid b \vee x = a \}$. We claim that $j$ is a minimal element of $X$. Indeed, if $j' < j$, then $j' \leq j_* \leq b$ by $b \wedge j = j_*$, hence $b \vee j' = b \neq a$. Since $X$ has the minimum element $\gamma(a \to b)$, we should have $j = \gamma (a \to b)$.
\end{proof}

Now we are ready to define the \emph{kappa map} $\kappa \colon \jirr L \to \mirr L$ by using these labelings, which plays a central role in this paper.
\begin{definition}\label{def:kappa}
  Let $L$ be a completely semidistributive lattice. Then define a map $\kappa \colon \jirr L \to \mirr L$ and $\kappa^d \colon \mirr L \to \jirr L$ as follows:
  \begin{align*}
    \kappa(j) &:= \mu(j \to j_*) = \max \{ x \in L \mid j \wedge x = j_* \}, \\
    \kappa^d(m) &:= \gamma(m^* \to m) = \min \{ x \in L \mid m \vee x = m^* \}.
  \end{align*}
\end{definition}
Now we have the following basic property, which says that $\kappa$ is bijective and two labelings $\gamma$ and $\mu$ coincide up to this bijection.
\begin{theorem}[{\cite[Proposition 9.2, Theorem 9.3]{thomas}}]\label{thm:kappa-bij}
  Let $L$ be a completely semidistributive lattice. Then the following hold.
  \begin{enumerate}
    \item $\kappa$ and $\kappa^d$ are mutually inverse bijections between $\jirr L$ and $\mirr L$.
    \item For each Hasse arrow $a \to b$, we have $\mu(a \to b) = \kappa(\gamma(a \to b))$.
    \item For every $j \in \jirr L$, we have $j \vee \kappa(j) = \kappa(j)^*$ and $j \wedge \kappa(j) = j_*$.
  \end{enumerate}
\end{theorem}
\begin{proof}
  (1)
  Let $j \in \jirr L$ and put $m := \kappa(j) = \mu(j \to j_*)$. Then Lemma \ref{lem:label-char}(2) implies $j \vee m = m^*$ and $j \wedge m = j_*$. This implies $j = \gamma(m^* \to m)$ by Lemma \ref{lem:label-char}(1), that is, $j = \kappa^d(m)$ holds. By duality, $\kappa \circ \kappa^d = \id_{\mirr L}$ holds.

  (2)
  Put $j = \gamma(a \to b) $ and $m = \mu(a \to b)$. Then Lemma \ref{lem:label-char} implies the following equalities:
  (i) $b \vee j = a$, (ii) $m \vee a = m^*$, (iii) $m \wedge a = b$, (iv) $b \wedge j = j_*$.
  Then (i), (ii), and $b \leq m$ imply $m \vee j = m \vee b \vee j = m \vee a = m^*$, and (iii), (iv), and $j \leq a$ implies $m \wedge j = m \wedge a \wedge j = b \wedge j = j_*$.
  Therefore, Lemma \ref{lem:label-char} implies $m = \gamma(j \to j_*)$, that is, $m = \kappa(j)$.

  (3) This follows from the proof of (1).
\end{proof}

\subsection{The kappa map in the lattice of torsion classes}
In this subsection, we recall the basics of lattice theory of torsion classes in an abelian length category and Barnard--Todorov--Zhu's result on the kappa map in \cite{BTZ}.

We begin with recalling the basic definitions. An \emph{abelian length category} is an abelian category such that every object has a composition series.
\begin{definition}
  Let $\AA$ be an abelian length category and $\CC$ a subcategory of $\AA$.
  \begin{enumerate}
    \item $\CC$ is \emph{closed under extensions} if, for any short exact sequence in $\AA$
    \[
    \begin{tikzcd}
      0 \rar & L \rar & M \rar & N \rar & 0,
    \end{tikzcd}
    \]
    we have that $L,N \in \CC$ implies $M \in \CC$.
    \item $\CC$ is \emph{closed under quotients (resp. subobjects) in $\AA$} if, for every object $C \in \CC$, any quotients (resp. subobjects) of $C$ in $\AA$ belong to $\CC$.
    \item $\CC$ is a \emph{torsion class (resp. torsion-free class) in $\AA$} if $\CC$ is closed under extensions and quotients in $\AA$ (resp. extensions and subobjects).
  \end{enumerate}
  For a collection $\CC$ of objects in $\AA$, we denote by $\TTT(\CC)$ (resp. $\FFF(\CC)$) the smallest torsion class (resp. torsion-free class) containing $\CC$.
  We denote by $\tors\AA$ and $\torf\AA$ the posets of torsion classes and torsion-free classes respectively, ordered by inclusion. If $\AA = \mod\Lambda$ for an artinian ring $\Lambda$, we simply write $\tors\Lambda$ and $\torf\Lambda$ instead of $\tors(\mod\Lambda)$ and $\torf(\mod\Lambda)$.
\end{definition}
Since $\tors\AA$ and $\torf\AA$ are closed under intersections, they are actually complete lattices with $\bigwedge X = \bigcap X$.

We have the following basic anti-isomorphism between $\tors\AA$ and $\torf\AA$. Here, for a collection $\CC$ of objects in $\AA$, we define its Hom-orthogonal subcategories $\CC^\perp$ and $^\perp\CC$ as follows:
\begin{align*}
  \CC^\perp & := \{ M \in \AA \mid \AA(C, M) = 0
  \text{ for every $C \in \CC$}\}, \\
  ^\perp \CC & := \{ M \in \AA \mid \AA(M, C) = 0
  \text{ for every $C \in \CC$}\}.
\end{align*}
\begin{proposition}
  Let $\AA$ be an abelian length category.
  Then we have the following anti-isomorphisms of complete lattices:
  \[
    \begin{tikzcd}[column sep = large]
      \tors\AA \rar[shift left, "(-)^\perp"]
      & \torf\AA \lar[shift left, "^\perp(-)"].
    \end{tikzcd}
  \]
\end{proposition}
Moreover, it is known that $\tors\AA$ is completely semidistributive:
\begin{theorem}[{\cite[Theorem 3.1(a)]{DIRRT}}]
  Let $\AA$ be an abelian length category. Then $\tors\AA$ is completely semidistributive.
\end{theorem}
Although it is assumed in \cite{DIRRT} that $\AA = \mod\Lambda$ for a finite-dimensional algebra $\Lambda$, the same proof applies for an abelian length category.

Therefore, we have a bijection $\kappa\colon \jirr(\tors\AA) \isoto \mirr(\tors\AA)$ by Theorem \ref{thm:kappa-bij}. Moreover, we have a description of $\jirr(\tors\AA)$ using \emph{bricks} in \cite{DIRRT, BCZ}, and a beautiful description of the bijection $\kappa$ is given in \cite{BTZ}.
In what follows, we explain their results and give (alternative) proofs in the setting of abelian length categories.

Let us begin with introducing \emph{bricks} and related notions.
\begin{definition}
  Let $\AA$ be an abelian length category.
  \begin{enumerate}
    \item An object $B$ of $\AA$ is called a \emph{brick} if $\End_\AA(B)$ is a division ring.
    \item For a subcategory $\CC$ of $\AA$, we denote by $\brick\CC$ the set of isomorphism classes of bricks contained in $\CC$.
    \item A \emph{semibrick in $\AA$} is a subset $\SS$ of $\brick\AA$ such that $\AA(B_1,B_2) = 0$ holds for $B_1 \neq B_2$.
    \item We denote by $\sbrick \AA$ the set of semibricks in $\AA$.
  \end{enumerate}
  We often identify an element of $\brick\AA$ with its representative.
\end{definition}

We need the following observation on the torsion closure of a brick later.
\begin{lemma}[{\cite[Lemma 4.4]{DIJ}}]\label{lem:tors-of-sbrick}
  Let $B$ be a brick in $\AA$. Then for every $M \in \TTT(B)$, every non-zero map $M \to B$ is a surjection.
\end{lemma}
\begin{proof}
  Consider the following subcategory $\CC$ of $\AA$:
  \[
    \CC := \{ M \in \AA \mid \text{every non-zero map $M \to B$ is surjective} \}
    \]
  We can easily check that $\CC$ is closed under quotients and extensions, thus $\CC$ is a torsion class. Moreover, $B \in \CC$ holds since $B$ is a brick. Thus $\TTT(B) \subseteq \CC$ holds by the minimality of $\TTT(B)$.
\end{proof}
Now we have the following relation between bricks, join-irreducibles, and meet-irreducibles.
\begin{theorem}\label{thm:btz}
  Let $\AA$ be an abelian length category. Then the following hold.
  \begin{enumerate}
    \item \cite[Theorem 3.3(c)]{DIRRT}, \cite[Theorem 1.5]{BCZ} We have a bijection $\brick\AA \isoto \jirr(\tors\AA)$ given by $B \mapsto \TTT(B)$.
    \item Dually, we have a bijection $\brick\AA \isoto \mirr(\tors\AA)$ given by $B \mapsto {}^\perp B$.
    \item \cite[Theorem 4.3.1]{BTZ} The composite $\jirr(\tors\AA) \to \mirr(\tors\AA)$ of bijections in \textup{(1)} and \textup{(2)} coincides with $\kappa$:
    \[
      \begin{tikzcd}[row sep = 0]
        \jirr(\tors\AA)  & \brick\AA \lar["\sim"'] \rar["\sim"] & \mirr(\tors\AA), \\
        \TTT(B) & \lar[mapsto] B \rar[mapsto] & {}^\perp B.
      \end{tikzcd}
    \]
    Namely, $\kappa(\TTT(B)) = {}^\perp B$ holds for every $B \in \brick\AA$.
  \end{enumerate}
\end{theorem}
\begin{proof}
  (1) First, we show the following claim:

  \textbf{(Claim)}: $\TT \subsetneq \TTT(B)$ if and only if $\TT \subseteq \TTT(B) \cap {}^\perp B$.

  The ``if'' part follows from $\TTT(B) \cap {}^\perp B \subsetneq \TTT(B)$ (this is because $B \in \TTT(B)$ and $B \not \in {}^\perp B$).
  Conversely, suppose $\TT \subsetneq \TTT(B)$. It suffices to show that every $M \in \TT$ satisfies $\AA(M,B)= 0$. If this is not the case, then there is some non-zero map $M \to B$, which is surjective by Lemma \ref{lem:tors-of-sbrick}.
  Thus we obtain $B \in \TT$ from $M \in \TT$. Hence $\TTT(B) \subseteq \TT$ holds, which is a contradiction.

  Now (Claim) implies that $\TTT(B)_* = \TTT(B) \cap {}^\perp B \subsetneq \TTT(B)$, thus $\TTT(B)$ is completely join-irreducible.
  Next, let $\TT \in \tors\AA$. Since we have $\TT = \Filt (\brick \TT)$ (see Lemma \ref{lem:heart-brick}), it follows that $\TT = \bigvee \{ \TTT(B) \mid B \in \brick \TT \}$. Therefore, if $\TT$ is completely join-irreducible, then $\TT= \TTT(B)$ for some brick $B$. This shows that the map $\TTT \colon \brick\AA \to \jirr(\tors\AA)$ is surjective.

  Finally, we will show that $\TTT(B_1) = \TTT(B_2)$ for $B_1,B_2 \in \brick\AA$ implies $B_1 \iso B_2$. If $B_2 \in {}^\perp B_1$, then $\TTT(B_2) \subseteq \TTT(B_1) \cap {}^\perp B_1 = \TTT(B_1)_* \neq \TTT(B_1)$, which contradicts $\TTT(B_1) = \TTT(B_2)$. Therefore, we obtain $B_2 \not \in {}^\perp B_1$, that is, there is a non-zero map $B_2 \to B_1$. Lemma \ref{lem:tors-of-sbrick} implies that this map is a surjection. In the same way, we obtain a surjection $B_1 \defl B_2$.
  Since $\AA$ is a length category, $B_1$ and $B_2$ should be isomorphic.

  (2)
  By considering the opposite abelian category $\AA^{\op}$ and using (1), we have a bijection $\brick\AA\isoto \jirr(\torf \AA)$ given by $B \mapsto \FFF(B)$. Consider the lattice anti-isomorphism $^\perp(-) \colon \torf\AA \to \tors\AA$. This sends $\FFF(B)$ to $^\perp \FFF(B) ={}^\perp B$ and induces a bijection $\jirr(\torf\AA) \iso \mirr(\tors\AA)$, hence the assertion holds.

  (3)
  Let $B \in \brick\AA$, and put $j:= \TTT(B) \in \jirr(\tors\AA)$ and $m:= {}^\perp B \in \mirr(\tors\AA)$ for simplicity.
  Recall that we have $j_* = \TTT(B) \cap {}^\perp B = j \wedge m$ by (Claim), and dually we have $\FFF(B)_* = \FFF(B) \cap B^\perp$ in $\torf\AA$.
  Thus by the lattice anti-isomorphism $^\perp(-) \colon \torf\AA \to \tors\AA$, we have the following equality in $\tors\AA$:
  \begin{align*}
    m^* &= ({}^\perp \FFF(B))^* = {}^\perp (\FFF(B)_*)
    = {}^\perp (\FFF(B) \cap B^\perp)\\
    &= {}^\perp \FFF(B) \vee {}^\perp(B^\perp)
    = {}^\perp B \vee \TTT(B) = m \vee j.
  \end{align*}
  Therefore, $j \wedge m = j_*$ and $j \vee m = m^*$ hold. This implies $\mu(j \to j_*) = m$ by Lemma \ref{lem:label-char}, that is, $\kappa(j) = m$.
\end{proof}

\section{Computing the posets of torsion hearts}\label{sec:3}
In this section, we introduce the notion of \emph{torsion hearts}, and prove our first main result Theorem \ref{thm:intro-b} (Theorem \ref{thm:main}).

We begin with the following standard construction in a poset.
\begin{definition}
  Let $P$ be a poset. We define the subset $\itv P$ of $P \times P$ as follows:
  \[
    \itv P := \{ (x,y) \in P \times P \mid x \leq y \text{ in $P$} \}.
  \]
  We often write $[x,y] \in \itv P$ instead of $(x,y) \in \itv P$.
\end{definition}
The symbol $\itv$ stands for \emph{intervals}, and we identify a pair $(x,y) \in \itv P$ with the closed interval $[x,y]$ in $P$. Note that we only consider closed intervals.

\subsection{Preliminaries on torsion hearts}
Next, we introduce the notion of \emph{torsion hearts}, which are subcategories associated with elements of $\itv(\tors\AA)$ (called the \emph{heart of intervals in $\tors\AA$} in e.g. \cite{ES} and the \emph{heart of twin torsion pairs} in \cite{tattar}).

\begin{definition}
  Let $\AA$ be an abelian length category.
  \begin{enumerate}
    \item For $[\UU, \TT] \in \itv(\tors\AA)$, define the subcategory $\HH_{[\UU,\TT]}$ of $\AA$ as follows:
    \[
      \HH_{[\UU, \TT]} := \TT \cap \UU^\perp.
    \]
    We call $\HH_{[\UU,\TT]}$ the \emph{heart} of the interval $[\UU,\TT]$.
    \item A subcategory $\CC$ is a \emph{torsion heart of $\AA$} if there is some $[\UU,\TT] \in \itv (\tors\AA)$ satisfying $\CC = \HH_{[\UU,\TT]}$.
  \end{enumerate}
  We denote by $\theart \AA$ the poset of torsion hearts in $\AA$ ordered by inclusion. If $\AA = \mod\Lambda$ for an artinian ring $\Lambda$, we simply write $\theart\Lambda$ instead of $\theart(\mod\Lambda)$.
\end{definition}

\begin{example}
  Every torsion class $\TT$ and torsion-free class $\FF$ is a torsion heart because $\TT = \HH_{[0,\TT]}$ and $\FF = ({}^\perp\FF)^\perp = \HH_{[{}^\perp\FF,\AA]}$. Thus $\tors\AA$ and $\torf\AA$ can be regarded as full subposets of $\theart\AA$.
\end{example}

\begin{remark}
  In general, $\theart\AA$ is not a lattice since $\theart\AA$ may not be closed under intersections.
  For example, $\theart kQ$ is not a lattice for the path algebra $kQ$ over a field $k$ of the quiver $Q \colon 1 \ot 2 \ot 3$. Indeed, $\add \{ \sst{1}, \sst{2\\1}, \sst{3\\2\\1}, \sst{3}\}$ and $\add \{ \sst{1}, \sst{3\\2\\1}, \sst{3\\2}, \sst{3}\}$ belong to $\theart kQ$, but a meet of them does not exist. 
\end{remark}

Recently, the author introduced \emph{ICE-closed subcategories} in \cite{eno-rigid} which generalize torsion classes and wide subcategories as follows. We say that a subcategory $\CC$ of an abelian category $\AA$ is \emph{closed under kernels (resp. cokernels, images)} if for every morphism $f \colon C_1 \to C_2$ in $\AA$ with $C_1, C_2 \in \CC$, we have $\ker f \in \WW$ (resp. $\coker f \in \WW$, $\im f \in \WW$).
\begin{definition}
  Let $\AA$ be an abelian length category and $\CC$ a subcategory of $\AA$.
  \begin{enumerate}
    \item $\CC$ is a \emph{wide subcategory} if it is closed under extensions, kernels, and cokernels.
    \item $\CC$ is \emph{ICE-closed} if it is closed under images, cokernels, and extensions.
  \end{enumerate}
  We denote by $\wide\AA$ and $\ice\AA$ the posets of wide subcategories and ICE-closed subcategories respectively, ordered by inclusion.
  If $\AA = \mod\Lambda$ for an artinian ring $\Lambda$, we simply write $\wide\Lambda$ and $\ice\Lambda$ instead of $\wide(\mod\Lambda)$ and $\ice(\mod\Lambda)$.
\end{definition}
Since $\wide\AA$ and $\ice\AA$ are closed under intersections, these posets are actually complete lattices.
By considering the opposite category $\AA^{\op}$, properties of \emph{IKE-closed} subcategories (subcategories closed under images, kernels, and extensions) follow from those of ICE-closed subcategories. Therefore, we omit statements about IKE-closed subcategories.

The following theorem claims that wide subcategories and ICE-closed subcategories are torsion hearts, and also give purely lattice-theoretical characterizations of intervals whose hearts are wide or ICE-closed.
\begin{theorem}\label{thm:itv-char}
  Let $\AA$ be an abelian length category.
  \begin{enumerate}
    \item \cite[Proposition 6.3]{AP} Every wide subcategory $\WW$ of $\AA$ is a torsion heart. Explicitly, $\WW$ is the heart of $[\TTT(\WW) \cap {}^\perp \WW, \TTT(\WW)]$.
    Moreover, $\HH_{[\UU,\TT]}$ is a wide subcategory for $[\UU,\TT] \in \itv(\tors\AA)$ if and only if the following equality holds in $\tors\AA$:
    \[
      \TT = \UU \vee \bigvee \{ \UU' \in \tors\AA \mid \UU \covd \UU' \leq \TT \}.
    \]
    \item \cite[Proposition 3.1, Theorem 3.4]{ES} Every ICE-closed subcategory of $\AA$ is a torsion heart. Moreover, $\HH_{[\UU,\TT]}$ is ICE-closed for $[\UU,\TT] \in \itv(\tors\AA)$ if and only if the following holds in $\tors\AA$:
    \[
      \TT \leq \UU \vee \bigvee \{ \UU' \in \tors\AA \mid \UU \covd \UU'\}.
    \]
  \end{enumerate}
\end{theorem}

Therefore, $\tors\AA$, $\torf\AA$, $\wide\AA$, and $\ice\AA$ are all full subposets of $\theart\AA$.
In what follows, we will compute $\theart\AA$, $\wide\AA$, and $\ice\AA$ from the lattice $\tors\AA$ using $\itv(\tors\AA)$ and tools developed in the previous section. The strategy is to use bricks, which can be represented by join-irreducibles in $\tors\AA$ by Theorem \ref{thm:btz}. The reason why this strategy works is due to the following fact.
We refer the reader to Definition \ref{def:filt} for the definition of $\Filt$.
\begin{lemma}\label{lem:heart-brick}
  Let $\CC$ be a torsion heart. Then $\CC = \Filt (\brick\CC)$ holds. In particular, for $\CC_1,\CC_2 \in \theart \AA$, we have $\CC_1 \subseteq \CC_2$ if and only if $\brick\CC_1 \subseteq \brick\CC_2$.
\end{lemma}
\begin{proof}
  The equality $\CC = \Filt(\brick\CC)$ is shown in \cite[Lemma 3.10]{DIRRT}. The ``only if" part of the remaining statement is clear. Conversely, if $\brick\CC_1 \subseteq \brick\CC_2$, then we have $\CC_1 = \Filt(\brick\CC_1) \subseteq \CC_2$ since $\CC_2$ is closed under extensions.
\end{proof}

\subsection{Construction}

In this subsection, we will construct some posets from a given completely semidistributive lattice $L$ such that these posets for $\tors\AA$ will be shown to be isomorphic to $\theart\AA$, $\wide\AA$, and $\ice\AA$.
\emph{Throughout this subsection, we denote by $L$ a completely semidistributive lattice.}
Although the constructions and their names are motivated by the lattice of torsion classes, we emphasize that \emph{all the constructions in this subsection only depends on the lattice structure of $L$}.

Let $\2^{\jirr L}$ denote the power set of $\jirr L$. Then $\2^{\jirr L}$ is a complete lattice. Recall that we have a bijection $\kappa \colon \jirr L \isoto \mirr L$, see Definition \ref{def:kappa} and Theorem \ref{thm:kappa-bij}.
\begin{definition}\label{def:jlabel}
  Let $L$ be a completely semidistributive lattice.
  Define a map $\jlabel \colon \itv L \to \2^{\jirr L}$ as follows:
  \[
    \jlabel [a, b] = \{ j \in \jirr L \mid \text{$j \leq b$ and $\kappa(j) \geq a$} \}.
  \]
  For a subset $\II$ of $\itv L$, we denote by $\jlabel \II$ the image of $\II$ under $\jlabel$, which we regard as a full subposet of $\2^{\jirr L}$, namely, $\jlabel\II$ is the poset of sets of completely join-irreducible elements of the form $\jlabel[a,b]$ for some $[a,b] \in \II$, ordered by inclusion.
\end{definition}
In Theorem \ref{thm:jlabel-label}, we will prove that $\jlabel[a,b]$ is precisely the set of join-irreducible labels appearing in the interval $[a,b]$. The notation $\jlabel$ is due to this fact.
In addition, as we shall see in the proof of Theorem \ref{thm:main}, the map $\jlabel$ is a combinatorial analogue of the map $\brick \HH_{(-)}$, that is, $\jlabel [a,b]$ models the set of bricks contained in the heart of $[a,b]$.

Next, we define \emph{wide intervals} and \emph{ICE intervals} of $L$, which correspond to intervals in $\tors\AA$ whose hearts are wide and ICE-closed by Theorem \ref{thm:itv-char}.
\begin{definition}\label{def:wide-ice-itv}
  Let $[a,b] \in \itv L$.
  \begin{enumerate}
    \item $[a,b]$ is a \emph{wide interval} if the following holds:
    \[
      b = a \vee \bigvee \{ a' \in L \mid a \covd a' \leq b \}.
    \]
    \item $[a,b]$ is an \emph{ICE interval} if the following holds:
    \[
      b \leq a \vee \bigvee \{ a' \in L \mid a \covd a' \}.
    \]
  \end{enumerate}
  We denote by $\witv L$ and $\iitv L$ the set of wide intervals and ICE intervals respectively.
\end{definition}

In this way, we obtain the posets $\jlabel(\itv L)$, $\jlabel(\witv L)$, and $\jlabel(\iitv L)$.

\subsection{Posets of subcategories as posets of join-irreducibles}

Now we can state our first main result of this paper.
\begin{theorem}\label{thm:main}
  Let $\AA$ be an abelian length category, and put $L:= \tors\AA$. Then we have the following isomorphisms of posets:
  \begin{enumerate}
    \item $\jlabel(\itv L) \iso \theart\AA$,
    \item $\jlabel(\witv L) \iso \wide\AA$,
    \item $\jlabel(\iitv L) \iso \ice\AA$.
  \end{enumerate}
  In particular, the posets $\theart\AA$, $\wide\AA$, and $\ice\AA$ can be computed only from the lattice structure of $L = \tors\AA$.
\end{theorem}

\begin{remark}
  In this theorem, $\wide\AA$ and $\ice\AA$ are not just posets but are complete lattices, but it is not clear a priori that $\jlabel(\witv L)$ and $\jlabel(\iitv L)$ are complete lattices.
\end{remark}

Before proving this theorem, we begin with the following easy but important observation.
\begin{lemma}\label{lem:kappa-torf}
  Let $\AA$ be an abelian length category, $B \in \brick\AA$, and $\UU \in \tors\AA$. Then we have $B \in \UU^\perp$ if and only if $\kappa(\TTT(B)) \geq \UU$ in $\tors\AA$.
\end{lemma}
\begin{proof}
  Since $\UU^\perp$ is a torsion-free class, we clearly have that $B \in \UU^\perp$ if and only if $\FFF(B) \leq \UU^\perp$ in $\torf\AA$. By the poset anti-isomorphism $^\perp(-) \colon \torf\AA \isoto \tors\AA$, this is equivalent to $^\perp \FFF(B) \geq {}^\perp (\UU^\perp) = \UU$.
  Now the assertion follows from $^\perp \FFF(B) = {}^\perp B = \kappa(\TTT(B))$ by Theorem \ref{thm:btz}.
\end{proof}

Now we are ready to prove Theorem \ref{thm:main}.

\begin{proof}[Proof of Theorem \ref{thm:main}]
  Recall that we have a bijection $\brick\AA \isoto \jirr L$ given by $B \mapsto \TTT(B)$ by Theorem \ref{thm:btz}. This bijection induces a poset isomorphism $\Phi \colon \2^{\brick\AA} \isoto \2^{\jirr L}$ between the power sets.
  Consider the following diagram, where $\HH_{(-)} \colon \itv (\tors\AA) \defl \theart \AA$ is given by taking hearts and $\brick \colon \theart\AA \to \2^{\brick\AA}$ is given by $\CC \mapsto \brick\CC$.
  \begin{equation}\label{eq:diagram}
    \begin{tikzcd}
      \itv L \dar["\jlabel"'] \rar["\HH_{(-)}", twoheadrightarrow]
      & \theart\AA \dar["\brick", hookrightarrow]\\
      \2^{\jirr L} & \2^{\brick\AA} \lar["\sim"', "\Phi"]
    \end{tikzcd}
  \end{equation}
  We first show that (\ref{eq:diagram}) is commutative. Let $[\UU,\TT] \in \itv L$. To prove $\Phi(\brick \HH_{[\UU,\TT]}) = \jlabel [\UU,\TT]$, it suffices to show that $B \in \HH_{[\UU,\TT]}$ for a brick $B \in \brick\AA$ if and only if $\TTT(B) \in \jlabel[\UU,\TT]$. This can be proved as follows:
  \begin{align*}
    B \in \HH_{[\UU,\TT]} & \Longleftrightarrow B \in \TT \cap \UU^\perp \\
    & \Longleftrightarrow B \in \TT \text{ and } B \in \UU^\perp \\
    & \Longleftrightarrow \TTT(B) \leq \TT \text{ and } \kappa(\TTT(B)) \geq \UU \\
    & \Longleftrightarrow \TTT(B) \in \jlabel [\UU, \TT]
  \end{align*}
  Here, the third equivalence follows from Lemma \ref{lem:kappa-torf}, and the last from the definition of the map $\jlabel$. Therefore, (\ref{eq:diagram}) is a commutative diagram.

  (1)
  In what follows, we always regard each subset of $\2^{\brick\AA}$ and $\2^{\jirr L}$ as a full subposet of them.
  The map $\HH_{(-)}$ is surjective by the definition of torsion hearts.
  Moreover, $\theart\AA \to \2^{\brick\AA}$ given by $\CC \mapsto \brick\CC$ is an injective poset embedding by Lemma \ref{lem:heart-brick}, thus $\theart\AA$ is isomorphic to its image $\brick(\theart\AA) \subseteq \2^{\brick \AA}$ as posets. Therefore, $\theart\AA$ is isomorphic to the image of $\brick \circ \HH_{(-)}$. Since $\Phi$ is a poset isomorphism, this in turn is isomorphic to the image of $\Phi \circ \brick \circ \HH_{(-)}$. Since (\ref{eq:diagram}) is commutative, the image of $\Phi \circ \brick \circ \HH_{(-)}$ coincides with the image of $\jlabel$, that is, $\jlabel(\itv L)$. Therefore, $\theart\AA$ is isomorphic to $\jlabel(\itv L)$ as posets.

  (2), (3)
  By $\wide\AA, \ice\AA \subseteq \theart\AA$, we have the following two diagrams similar to (\ref{eq:diagram}).
  \[
    \begin{tikzcd}
      \witv L \dar["\jlabel"'] \rar["\HH_{(-)}", twoheadrightarrow]
      & \wide\AA \dar["\brick", hookrightarrow]\\
      \2^{\jirr L} & \2^{\brick\AA} \lar["\sim"', "\Phi"]
    \end{tikzcd}
    \qquad
    \begin{tikzcd}
      \iitv L \dar["\jlabel"'] \rar["\HH_{(-)}", twoheadrightarrow]
      & \ice\AA \dar["\brick", hookrightarrow]\\
      \2^{\jirr L} & \2^{\brick\AA} \lar["\sim"', "\Phi"]
    \end{tikzcd}
  \]
  Here in each case, $\HH_{(-)}$ is well-defined and surjective by Theorem \ref{thm:itv-char} and Definition \ref{def:wide-ice-itv}. Therefore, in the same way as to (1), we can prove that $\wide\AA$ and $\ice\AA$ are isomorphic to $\jlabel(\witv L)$ and $\jlabel (\iitv L)$ as posets respectively.
\end{proof}

\begin{example}
  Let $k$ be a field, $Q$ be the quiver $1 \to 2$, and $kQ$ be the path algebra of $Q$. Then the Hasse quiver of $L := \tors kQ$ is as follows, where we also show its join-irreducible labeling.
  
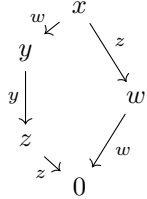
\[
    \begin{tikzcd}[sep = tiny]
      & x \ar[ld, "w"'] \ar[rdd, "z"] \\
      y \ar[dd, "y"'] \\
      & & w \ar[ldd, "w"] \\
      z \ar[rd, "z"'] \\
      & 0
    \end{tikzcd}
  \]
  By simple computation, we have $\jirr L = \{ y, z, w\} = \mirr L$ and $\kappa(y) = z$, $\kappa(z) = w$, $\kappa(w) = y$.
  To compute $\jlabel [z, x]$ for example, we check which $j \in \jirr L$ satisfies $j \leq x$ and $\kappa(j) \geq z$, and we obtain $\jlabel [z,x] = \{y, w\}$. Instead, we can use Theorem \ref{thm:jlabel-label} to compute $\jlabel$ using the join-irreducible labeling.
  \begin{table}[htp]
    \begin{tabular}{c|c}
      Subsets $\II$ of $\itv L$ & $\jlabel \II$ \\ \hline \hline
      \makecell{
        $\itv L = \{ [0,x],[0,y], [0,z], [0,w], [0,0], [w,x], $\\
        $[w,w], [z,x], [z,y], [z,z], [y, x], [y,y], [x,x]\}$
      } &
      $\{\varnothing, \{y\}, \{z\}, \{w\}, \{y, z\}, \{y, w\}, \{y, z, w\} \}$ \\ \hline
      \makecell{
        $\witv L = \{ [0,x], [0,z], [0,w], [0,0], [w,x], $\\
        $[w,w], [z,y], [z,z], [y, x], [y,y], [x,x]\}$
      } &
      $\{ \varnothing, \{y\}, \{z\}, \{w\}, \{y,z,w\} \}$ \\ \hline
      \makecell{
        $\iitv L = \{ [0,x],[0,y], [0,z], [0,w], [0,0], [w,x], $\\
        $[w,w], [z,y], [z,z], [y, x], [y,y], [x,x]\}$
      } &
      $\{ \varnothing, \{y\}, \{z\}, \{w\}, \{y,z\}, \{y, z, w\} \}$
    \end{tabular}
    \caption{Sets of intervals and their images under $\jlabel$}
    \label{tab:a2itv}
  \end{table}
  In Table \ref{tab:a2itv}, we list $\itv L$, $\witv L$, and $\iitv L$, and their images under $\jlabel \colon \itv L \to \2^{\jirr L}$. Hence the second column (viewed as the posets ordered by inclusion) gives the posets isomorphic to $\theart\Lambda$, $\wide\Lambda$, and $\ice\Lambda$ respectively.
\end{example}
See Example \ref{ex:intro-alg} in the introduction for more examples.

\subsection{The map \texorpdfstring{$\jlabel$}{jlabel} in terms of the join-irreducible labeling}

In this subsection, we give a more intuitive description of the map $\jlabel \colon \itv L \to \2^{\jirr L}$ for a completely semidistributive lattice: $\jlabel [a,b]$ is precisely the set of all join-irreducible labels appearing in $[a,b]$.
\begin{theorem}\label{thm:jlabel-label}
  Let $L$ be a completely semidistributive lattice and $[a,b] \in \itv L$. Then we have the following equality:
  \[
    \jlabel [a,b] = \{ \gamma(x \to y) \mid a \leq y \covd x \leq b \},
  \]
  where $\gamma \colon \Hasse_1 L \to \jirr L$ is the join-irreducible labeling  (Definition \ref{def:label}).
\end{theorem}

To prove this, we need the following lemma. This is a lattice-theoretical analogue of \cite[Theorem 3.4]{DIRRT}.

\begin{lemma}\label{lem:itv-label}
  Let $L$ be a completely semidistributive lattice, $x \in L$, and $j \in \jirr L$. Suppose that $x \leq \kappa(j)$ holds. Then there is a Hasse arrow $x \vee j \to (x \vee j) \wedge \kappa(j)$ in $L$, and its join-irreducible label is $j$.
\end{lemma}

\begin{proof}
  Put $m := \kappa(j)$ for simplicity. We first show that there is a Hasse arrow $p \colon x \vee j \to (x \vee j) \wedge m$ in $L$.
  Suppose that $(x \vee j) \wedge m \leq y \leq x \vee j$ holds for $y \in L$. Then by applying $j \wedge (-)$, we obtain $j \wedge m \leq j \wedge y \leq j$. Since we have $j \wedge m = j_*$ by the definition of $m = \kappa(j)$, we must have either $j \wedge y = j \wedge m$ or $j \wedge y = j$.
  If we have $j \wedge y = j \wedge m$, then $y \leq j \wedge y = j \wedge m \leq m$ holds, thus $y \leq (x \vee j) \wedge m$ holds, which implies $y = (x \vee j) \wedge m$. On the other hand, if we have $j \wedge y = j$, then $j \leq y$ holds. In addition, we have $x \leq m$ by the assumption, hence $x \leq (x\vee j) \wedge m \leq y$ holds. Therefore, we obtain $x \vee j \leq y$, which implies $y = x \vee j$.

  Next, we check $\gamma(p) = j$. To this aim, we first compute the meet-irreducible label. We have $(x \vee j) \vee m = j \vee m = m^*$ by $x \leq m$ and Theorem \ref{thm:kappa-bij}. Thus Lemma \ref{lem:label-char} implies $\mu(p) = m$. Therefore, Theorem \ref{thm:kappa-bij} implies $\gamma(p) = \kappa^{d}(m) = j$.
\end{proof}

\begin{proof}[Proof of Theorem \ref{thm:jlabel-label}]
  Denote by $J$ the right hand side. We will prove $\jlabel [a,b] = J$.

  Let $j \in J$. Then there is a Hasse arrow $x \to y$ satisfying $\gamma(x \to y) = j$ and $a \leq y < x \leq b$. By the definition of $\gamma$, we have $y \vee j = x$, hence $j \leq y \vee j = x \leq b$ holds. On the other hand, Theorem \ref{thm:kappa-bij} implies $\mu(x \to y) = \kappa(j)$, hence Lemma \ref{lem:label-char} implies $x \wedge \kappa(j) = y$. Therefore, we obtain $a \leq y = x \wedge \kappa(j) \leq \kappa(j)$. Thus we have $j \in \jlabel[a,b]$.

  Conversely, let $j \in \jlabel [a,b]$, that is, $j \in \jirr L$, $j \leq b$, and $a \leq \kappa(j)$ hold. Then Lemma \ref{lem:itv-label} implies that there is a Hasse arrow $p \colon a \vee j \to (a \vee j) \wedge \kappa(j)$ whose join-irreducible label is $j$. Note that we have $a \leq (a \vee j) \wedge \kappa(j)$ by $a \leq \kappa(j)$ and $a \vee j \leq b$ by $j \leq b$. Hence $p$ lies inside $[a,b]$, which shows $j = \gamma(p) \in J$.
\end{proof}

\section{The kappa order, the core label order, and \texorpdfstring{$\wide\AA$}{wide A}}\label{sec:4}
In this section, we will give two more descriptions of $\wide\AA$.
More precisely, for a given completely semidistributive lattice $L$, we first consider the set $L_0$ of elements with \emph{canonical join representations}. Then we provide two poset structures on $L_0$: the \emph{kappa order $\leq_\kappa$} and the \emph{core label order} $\leq_\clo$, where the former is defined using the \emph{extended kappa map} introduced by Barnard--Todorov--Zhu, and the latter using the join-irreducible labeling. 
Then we will show that these posets are isomorphic to $\wide\AA$ for $L = \tors\AA$. In particular, two orders coincide in this case, which is not true even for finite congruence-uniform lattices (Example \ref{ex:not-same}).

The core label order is a generalization of the poset structure on finite congruence-uniform lattices (also known as the \emph{shard intersection order}) introduced by Reading \cite{reading-shard,reading-region} and studied by several papers such as \cite{muhle,biclosed,cjc-biclosed}.
Thus our result can be regarded as another characterization of the core label order using the extended kappa map when $L$ is isomorphic to $\tors\AA$ for some abelian length category $\AA$. This class contains important two classes: the weak order of the finite Weyl group (the resulting core label order is Reading's original shard intersection order), and the Cambrian lattice of simply-laced Dynkin type (the resulting core label order is isomorphic to the lattice of non-crossing partitions).
In Section \ref{sec:combi}, we will see some consequences for these classes.

\subsection{Canonical join representation and the extended kappa map}
In this subsection, we introduce canonical join representations and explain the extended kappa map given in \cite{BTZ}, and study their basic properties.
\begin{definition}\label{def:cjr}
  Let $L$ be a complete lattice and $x$ an element of $L$.
  \begin{enumerate}
    \item A \emph{join representation of $x$} is an expression of the form $x = \bigvee A$ for some subset $A \subseteq L$.
    \item Let $x = \bigvee A = \bigvee B$ be two join representations. We say that \emph{$A$ refines $B$} if for every $a \in A$ there is some $b \in B$ with $a \leq b$.
    \item A join representation $x = \bigvee A$ is a \emph{canonical join representation} if it satisfies the following conditions:
    \begin{enumerate}
      \item $A$ refines \emph{every} join representation of $x$, that is, if $x = \bigvee B$, then $A$ refines $B$.
      \item $A$ is an antichain, that is, $a_1 \leq a_2$ with $a_1, a_2 \in A$ implies $a_1 = a_2$.
    \end{enumerate}
  \end{enumerate}
  It is easily checked that a canonical join representation of $x \in L$ is unique if it exists. We denote by $\CJR(x) := A$ if $x$ has a canonical join representation $x = \bigvee A$. We define a subset $L_0$ of $L$ as follows:
  \[
    L_0 = \{ x \in L \mid \text{$\CJR(x)$ exists} \}.
  \]
  Dually, we define \emph{canonical meet representation}, $\CMR(x)$ for $x \in L$ and
  \[
    L^0 = \{ x \in L \mid \text{$\CMR(x)$ exists} \}.
  \]
\end{definition}

\begin{remark}
  Consider the following condition for a join representation $x = \bigvee A$:
  \begin{itemize}
    \item[(b)$'$] $x = \bigvee A$ is \emph{irredundant}, that is, for every proper subset $A' \subsetneq A$, we have $\bigvee A' < x$.
  \end{itemize}
  Then it is easily checked that $x = \bigvee A$ satisfies (a) and (b) if and only if it satisfies (a) and (b)$'$.
\end{remark}

\begin{remark}\label{rem:CJR-def}
  Let us emphasize some ambiguity in the definition of canonical join representations for infinite lattices in the literature.
  This version of definition seems to be standard in lattice theory, e.g. \cite{AN, free-lattice, Gor, JR, reading-region, RST}.
  On the other hand, our definition is different from the definition in \cite{BCZ,BTZ} for the infinite case. In these papers, a join representation $x = \bigvee A$ is called a canonical join representation if
  \begin{itemize}
    \item[(b)$'$] $x = \bigvee A$ is irredundant.
    \item[(a)$'$] If $x = \bigvee B$ is an \emph{irredundant} join representation, then $A$ refines $B$. 
  \end{itemize}
  If $L$ is a finite lattice, then every join representation is refined by some irredundant join representation (by removing unnecessary elements), thus our definition coincides with theirs.
  However, if $L$ is an infinite lattice, a join representation is canonical in their sense if it is so in our sense, but the converse fails.
  For example, $x = \bigvee \{ x \}$ is a canonical join representation in our sense if and only if $x$ is \emph{completely} join-irreducible, while it is so in their definition if and only if $x$ is join-irreducible.
  As we will see in Theorems \ref{thm:cjr-sbrick} and \ref{thm:ext-kappa-bij}, when we consider the extended kappa map and canonical join representations and study their relation to semibricks and wide subcategories, our definition seems to be more suitable.
\end{remark}

There is the following characterization of elements with canonical join representations due to Gorbunov, which we shall need later.
\begin{proposition}[{\cite[Theorem 1]{Gor}}]\label{prop:cjr-char}
  Let $L$ be a completely semidistributive lattice and $x \in L$. Then the following are equivalent.
  \begin{enumerate}
    \item $x$ has a canonical join representation.
    \item For every $y \in L$ with $y < x$, there exists $x' \in L$ satisfying $y \leq x' \covd x$.
  \end{enumerate}
  In particular, if $L$ is finite, then every element has a canonical join representation.
\end{proposition}

We can easily check that each element appearing in canonical join representations  is completely join-irreducible:
\begin{lemma}\label{lem:joinand-irred}
  Let $L$ be a complete lattice and $x = \bigvee A$ is a canonical join representation. Then each $a \in A$ is completely join-irreducible.
\end{lemma}
\begin{proof}
  Suppose $a = \bigvee B$ for $B \subseteq L$. Then we have
  $x = a \vee \bigvee (A \setminus \{ a \}) = \bigvee \big(B \cup (A \setminus \{a\}) \big)$. Since $A$ refines every join representation, we have $a \leq b$ for some $b \in B$ or $b \in A \setminus \{ a\}$, but the latter is impossible since $A$ is an antichain. Thus $b \in B$, which implies $a \leq b \leq \bigvee B = a$, namely, $a = b \in B$. Thus $a$ is completely join-irreducible.
\end{proof}
By this property, we can extend the kappa map $\kappa \colon \jirr L \to \mirr L$ in a completely semidistributive lattice as follows.
\begin{definition}[{\cite[Definition 1.1.3]{BTZ}}]
  Let $L$ be a completely semidistributive lattice. Then we define the \emph{extended kappa map $\exk \colon L_0 \to L$} as follows:
  \[
    \exk (x) = \bigwedge \{ \kappa(j) \mid j \in \CJR(x) \}.
  \]
\end{definition}
If $j$ is completely join irreducible, then $j = \bigvee\{j\}$ is a canonical join representation, hence we have $\exk(j) = \kappa (j)$.
Note that due to the difference of the definition of canonical join representations between this paper and \cite{BTZ} (Remark \ref{rem:CJR-def}), the domain of $\exk$ differ from theirs.

We will need some properties of canonical join representations later.
First, the following simple observation is quite useful.
\begin{lemma}[{\cite[Lemma 2.57]{free-lattice}}]\label{lem:kappa}
  Let $L$ be a completely semidistributive lattice, $x \in L$, and $j \in \jirr L$. Then $x \leq \kappa(j)$ if and only if $x \vee j_* \neq x \vee j$ holds.
\end{lemma}
\begin{proof}
  Suppose that $x \leq \kappa(j)$ holds. Then if $x \vee j_* = x \vee j$, then $j \leq x \vee j = x \vee j_* \leq \kappa(j) \vee j_* = \kappa(j)$, which is a contradiction, hence $x \vee j_* \neq x \vee j$ holds.

  Conversely, suppose that $x \vee j_* \neq x \vee j$ holds. Then we have $j \not \leq x \vee j_*$. On the other hand, $j_* \leq (x \vee j_*) \wedge j \leq j$, and $(x \vee j_*) \wedge j \neq j$ holds. Hence $(x \vee j_*) \wedge j = j_*$ holds, which implies $x \leq x \vee j_* \leq \kappa(j)$ by the definition of $\kappa(j)$.
\end{proof}
We have the following necessary condition for a set of completely join-irreducible elements to form a canonical join representation.
\begin{lemma}[{c.f. \cite[Theorem 5.13]{RST}}]\label{lem:join-ortho}
  Let $L$ be a completely semidistributive lattice and $x = \bigvee A$ a canonical join representation. Then for every $i, j \in A$ with $i \neq j$, we have $i \leq \kappa(j)$.
\end{lemma}
\begin{proof}
  Observe that $j \in \jirr L$ holds for $j \in A$ by Lemma \ref{lem:joinand-irred}, hence $\kappa(j)$ is defined. Suppose that there are $i, j \in A$ with $i \neq j$ satisfying $i \not\leq \kappa(j)$.
  Then Lemma \ref{lem:kappa} implies $i \vee j = i \vee j_*$, thus we have the following join representation:
  \[
    x = j_* \vee \bigvee (A \setminus \{ j \})
  \]
  Since $x = \bigvee A$ is the canonical join representation, it refines the above join representation. Therefore, $j \leq j_*$ or $j \leq a$ for some $a \in A \setminus \{j\}$. Since the former is impossible, $j \leq a$ for some $a \in A$ with $a \neq j$. This contradicts the fact that $A$ is an antichain.
\end{proof}

We have the following converse of the above lemma. This generalizes \cite[Theorem 5.13]{RST} where $L$ is assumed to be finite.
\begin{proposition}\label{prop:cjr-partial-inverse}
  Let $L$ be a completely semidistributive lattice and $x = \bigvee A$ a join representation with $A \subseteq \jirr L$. Suppose that $i \leq \kappa(j)$ holds for every $i,j \in A$ with $i\neq j$. If $x$ has a canonical join representation, then $\CJR(x) = A$ holds.
\end{proposition}
\begin{proof}
  Let $x = \bigvee B$ be a canonical join representation, and we will show $A = B$. Suppose that $A \not \subseteq B$ holds, and take $a \in A$ with $a \not \in B$. Then since $B$ refines $A$, every $b \in B$ satisfies that either $b \leq a$ or $b \leq \bigvee (A \setminus \{a\})$, and the former is equivalent to $b \leq a_*$ by $a \not\in B$. Therefore, we have the following inequality:
  \[
    x = \bigvee B \leq a_* \vee \bigvee (A \setminus \{a\})
    \leq a \vee \bigvee(A \setminus \{a\}) = \bigvee A = x,
  \]
  hence we obtain $a_* \vee \bigvee(A \setminus \{a\}) = a \vee \bigvee (A \setminus \{a\})$. Then Lemma \ref{lem:kappa} implies $\bigvee (A \setminus \{a\}) \not\leq \kappa(a)$. Therefore, there is some $a' \in A \setminus \{a\}$ satisfying $a' \not\leq \kappa(a)$, which is a contradiction. Hence $A \subseteq B$ holds. Let $b \in B$. Then $b \leq a$ for some $a \in A \subseteq B$, and since $B$ is an antichain, we must have $b = a$. Thus $A = B$ holds.
\end{proof}

The following interpretation of the canonical join representation in terms of the join-irreducible labeling is useful.
For an element $x$ in a completely semidistributive lattice, we denote by $\jlabel_\downarrow x$ (resp. $\jlabel^\uparrow x$) the set of join-irreducible labels of Hasse arrows starting at $x$ (resp. ending at $x$).
This result is shown in \cite[Lemma 19]{cjc} when $L$ is finite.
\begin{lemma}\label{lem:cjr-label}
  Let $L$ be a completely semidistributive lattice and $x \in L$, and suppose that $x$ has a canonical join representation. Then $\CJR(x) = \jlabel_\downarrow x$ holds.
\end{lemma}
\begin{proof}
  Suppose that $x$ has a canonical join representation $x = \bigvee J$.
  It is shown in \cite[Section 5]{Gor} that there is a bijection
  \[
    \begin{tikzcd}[row sep= 0]
      J \rar["\sim"] & \{ y \in L \mid x \cov y\}, \\
      j \rar[mapsto] & x_j := \bigvee \{ a \in L \mid j \not\leq a \leq x \}.
    \end{tikzcd}
  \]
  Therefore, it is enough to show that $y \mapsto \gamma(x \to y)$ is an inverse of this bijection, that is, $\gamma(x \to x_j) = j$.
  Since $J$ is an antichain, $J \setminus \{j\}$ is a subset of $\{ a \in L \mid j \not\leq a \leq x \}$. Therefore, we have $x = x_j \vee j$.
  Suppose that $x_j \vee z = x$, and it suffices to show $j \leq z$ in order to prove $\gamma(x \to x_j) = j$. Indeed, if $j \not \leq z$, then $z \leq x_j$ by the definition of $x_j$. Therefore, $x_j \vee z = x_j \neq x$, which is a contradiction.
\end{proof}

This gives the following description of the extended kappa map in terms of the join-irreducible labeling at least when $L$ is finite.
\begin{corollary}\label{cor:ext-kappa-char}
  Let $L$ be a finite semidistributive lattice. Then for each $x \in L$, there is a unique element $y$ satisfying $\jlabel_\downarrow x = \jlabel^\uparrow y$, and in this case, $y = \exk(x)$ holds. Moreover, $\exk \colon L \to L$ is bijective.
\end{corollary}
\begin{proof}
  Since $L$ is finite, every element has a canonical join representation and a canonical meet representation by Proposition \ref{prop:cjr-char}, hence $L = L_0 = L^0$.
  Let $x \in L$. Then Lemma \ref{lem:cjr-label} shows $\CJR(x) = \jlabel_\downarrow x$, and Lemma \ref{lem:join-ortho} implies that $i \leq \kappa(j)$ holds for $i,j \in \CJR(x)$ with $i \neq j$.

  Put $y:= \exk(x) = \bigwedge \{\kappa(j) \mid j \in \CJR(x)\}$. Since $y$ has a canonical meet representation, the dual of Lemma \ref{prop:cjr-partial-inverse} implies $\CMR(y) = \{\kappa(j) \mid j \in \CJR(x)\}$. Hence the dual of Lemma \ref{lem:cjr-label} together with Theorem \ref{thm:kappa-bij}(2) implies $\jlabel^\uparrow y = \CJR(x) = \jlabel_\downarrow x$.

  Conversely, suppose that $y \in L$ satisfies $\jlabel^\uparrow y = \jlabel_\downarrow x$. Since $y$ has a canonical meet representation, $\CMR(y) = \{\kappa(j) \mid j \in \jlabel_\downarrow x\}$ holds by the dual of Lemma \ref{lem:cjr-label} and Theorem \ref{thm:kappa-bij}. In particular, $y = \bigwedge \CMR(y) = \exk(x)$ holds.
  The proof of the last statement is clear from the above argument, hence we omit it.
\end{proof}

\begin{example}
  Let $L$ be a lattice in Example \ref{ex:intro-alg}, whose join-irreducible labels are shown in Figure \ref{fig:label-ex}. Consider $\exk(\ov{5})$. Then we have $\jlabel_\downarrow \ov{5} = \{ 2, 3\}$. Therefore, $\exk(\ov{5})$ is the unique element $y$ satisfying $\jlabel^\uparrow y = \{2, 3\}$ by Corollary \ref{cor:ext-kappa-char}, and we can find that $1$ is such an element. Thus $\exk(\ov{5}) = 1$ holds.
\end{example}

We will see later in Theorem \ref{thm:ext-kappa-bij} that $\exk$ gives a bijection between $L_0$ and $L^0$ when $L = \tors \AA$ for an abelian length category $\AA$. Thus we have the following natural question.
\begin{question}\label{q:1}
  Let $L$ be a completely semidistributive lattice. Then does $\exk(x)$ have a canonical meet representation for $x \in L_0$? If so, then $\exk$ gives a bijection $L_0 \isoto L^0$ by the same argument as in Corollary \ref{cor:ext-kappa-char}
\end{question}

\subsection{Canonical join representation and widely generated torsion classes}
The aim of this section is to describe the relation between widely generated torsion classes and canonical join representations, and to explain a representation-theoretic interpretation of the extended kappa map given in \cite{BTZ}.
We begin with introducing the related notions.
\begin{definition}\label{def:filt}
  Let $\AA$ be an abelian length category.
  \begin{enumerate}
    \item A torsion class $\TT$ in $\AA$ is \emph{widely generated} \cite{AP} if there exists some wide subcategory $\WW$ of $\AA$ satisfying $\TT = \TTT(\WW)$.
    \item For a class $\CC$ of objects in $\AA$, we denote by $\Filt\CC$ the subcategory of $\AA$ consisting of $M \in \AA$ such that there is a filtration $0 = M_0 \subseteq M_1 \subseteq \cdots \subseteq M_n = M$ of subobjects of $M$ satisfying $M_i/M_{i-1} \in \CC$ for each $i$.
    \item For a wide subcategory $\WW$ of $\AA$, we denote by $\simp\WW$ the set of isomorphism classes of simple objects in an abelian category $\WW$.
    \item For a torsion class $\TT$ in $\AA$, we define the subcategory $\WL(\TT)$ of $\AA$ as follows:
    \[
      \WL(\TT) = \{ W \in \TT \mid \text{$\ker f \in \TT$ for every $f \colon T \to W$ with $T \in \TT$} \}.
    \]
  \end{enumerate}
\end{definition}

Since completely join-irreducible elements of $\tors\AA$ can be described by bricks, one can consider the following map $\ov{\CJR}$ instead of $\CJR$:
\begin{definition}
  Let $\AA$ be an abelian length category, and suppose that $\TT \in \tors\AA$ has a canonical join representation in $\tors\AA$. Then we define a set $\ov{\CJR}(\TT)$ of bricks as follows: Consider $\CJR(\TT)$, which is a set of completely join-irreducible elements by Lemma \ref{lem:joinand-irred}. Under the bijection between completely join-irreducible elements and bricks in Theorem \ref{thm:btz}, we obtain a set of bricks $\ov{\CJR}(\TT)$ corresponding to $\CJR(\TT)$.
\end{definition}
Next, we recall some results related to wide subcategories and torsion classes.
\begin{proposition}\label{prop:wide-related}
  Let $\AA$ be an abelian length category.
  \begin{enumerate}
    \item \cite[1.2]{ringel} $\Filt \colon \sbrick\AA \to \wide\AA$ and $\simp \colon \wide\AA \to \sbrick\AA$ are mutually inverse bijections.
    \item \cite[Proposition 3.3]{MS} $\WL(\TT) \in \wide\AA$ holds for $\TT \in \tors\AA$, thus we have a map $\WL \colon \tors\AA \allowbreak\to \wide\AA$. Moreover, the composition $\WL \circ \TTT \colon \wide\AA \to \tors\AA \to \wide\AA$ is the identity.
    \item \cite[Theorem 7.2]{AP} A torsion class $\TT$ is widely generated if and only if $\TT$ satisfies the following condition: for every $\UU \in \tors\AA$ with $\UU \subsetneq \TT$, there exists $\TT' \in \tors\AA$ satisfying $\UU \subseteq \TT' \covd \TT$.
  \end{enumerate}
\end{proposition}
Since $\TTT(\SS) = \TTT(\Filt\SS)$ holds for a semibrick $\SS$, the above first result implies that a torsion class $\TT$ is widely generated if and only if there is a semibrick $\SS$ satisfying $\TT = \TTT(\SS)$. We also note that $\TTT(\SS) = \bigvee \{ \TTT(B) \mid B \in \SS \}$ holds.

Now we are ready to prove the following relation between semibricks, wide subcategories, and torsion classes with canonical join representations. We note that this result is implicitly given in \cite[Section 3.2, Corollary 5.1.8]{BCZ}, but since their definition of canonical join representations is different from ours (see Remark \ref{rem:CJR-def}), they did not state it in this form (c.f. \cite[Remark 4.4.10]{BTZ}).
We shall give two proofs: a new lattice-theoretic proof, and a representation-theoretic proof which is essentially in \cite{BTZ}. This will help us understand the relation between lattice theory and representation theory of algebras.
\begin{theorem}\label{thm:cjr-sbrick}
  Let $\AA$ be an abelian length category.
  \begin{enumerate}
    \item $\TT \in \tors\AA$ is widely generated if and only if $\CJR(\TT)$ in $\tors\AA$ exists.
    \item $\TTT \colon \wide\AA \to (\tors\AA)_0$ and $\WL \colon (\tors\AA)_0 \to \wide\AA$ are mutually inverse bijections.
    \item $\ov{\CJR}(\TT)$ is a semibrick for $\TT \in (\tors\AA)_0$, and $\TTT \colon \sbrick\AA \to (\tors\AA)_0$ and $\ov{\CJR} \colon (\tors\AA)_0 \to \sbrick\AA$ are mutually inverse bijections.
    \item We have the following commutative diagram consisting of bijections.
    \[
      \begin{tikzcd}[column sep = large]
        \sbrick \AA \dar[shift left, "\Filt"] \rar["\TTT", shift left]
        & (\tors\AA)_0 \dar[equal] \lar["\ov{\CJR}", shift left]\\
        \wide\AA \rar["\TTT", shift left] \uar["\simp", shift left]
        & (\tors\AA)_0 \lar["\WL", shift left]
      \end{tikzcd}
    \]
  \end{enumerate}
\end{theorem}
\begin{proof}
  We provide two different proofs of (1) and (3). The first one is based on the lattice theoretic observations in \cite{AP,Gor}, and the second on the representation theoretic observations in \cite{BCZ}.

  \noindent
  \emph{First proof.}
  (1)
  This follows from Proposition \ref{prop:wide-related}(3) and Proposition \ref{prop:cjr-char}.

  (2)
  Since $\WL \circ \TTT \colon \wide\AA \to \tors\AA \to \wide\AA$ is the identity by Proposition \ref{prop:wide-related}(2), it follows that $\TTT$ and $\WL$ induces bijections between $\wide\AA$ and the image of $\TTT \colon \wide\AA \to \tors\AA$. Then (2) follows from (1).

  (3)
  Here we only prove that $\ov{\CJR}(\TT)$ is a semibrick for $\TT \in (\tors\AA)_0$. The remaining assertion will then follow from the commutativity $\ov{\CJR} = \simp \circ \WL$, which will be proved in (4).

  Let $B,C \in \ov{\CJR}(\TT)$ with $B$ and $C$ non-isomorphic. By Lemma \ref{lem:join-ortho}, we have that $\TTT(B) \leq \kappa(\TTT(C))$ holds in $\tors\AA$. Since $\kappa(\TTT(C)) = {}^\perp C$ by Theorem \ref{thm:btz}, we have $\TTT(B) \subseteq {}^\perp C$, hence $B \in {}^\perp C$. This shows that $\AA(B,C) = 0$, hence $\ov{\CJR}(\TT)$ is a semibrick.

  (4)
  We only have to show $\simp\WL(\TT) = \ov{\CJR}(\TT)$ holds for $\TT \in (\tors\AA)_0$, since the other commutativity is clear.
  By \cite[Theorem 6.7]{AP}, we have that $\simp\WL(\TT)$ is the set of \emph{brick labels} of Hasse arrows starting at $\TT$. Here, we omit the definition of the brick labeling, but this labeling is compatible with the join-irreducible labeling by \cite[Theorem 3.11]{DIRRT} under the bijection in Theorem \ref{thm:btz}. Thus $\simp\WL(\TT) = \{ B \in \brick \AA \mid \TTT(B) \in \jlabel_\downarrow \TT\}$ holds. Then Lemma \ref{lem:cjr-label} implies $\simp\WL(\TT) = \ov{\CJR}(\TT)$.

  \noindent
  \emph{Second proof.}
  (1)
  Suppose that $\TT$ is widely generated. This means that $\TT = \TTT(\SS)$ holds for some semibrick $\SS$. Thus $\TT = \bigvee \{ \TTT(B) \mid B \in \SS \}$ holds. The fact that this is a canonical join representation of $\TT$ is shown in \cite[Proposition 3.7]{BCZ} (one can easily check that their proof works also for our definition of canonical join representations).
  Conversely, suppose that $\TT$ has a canonical join representation. Then it should be of the form $\TT = \bigvee \{ \TTT(B) \mid B \in \SS \}$ for some set $\SS$ of bricks by Lemma \ref{lem:joinand-irred} and Theorem \ref{thm:btz}. Then the fact that $\SS$ should be a semibrick is shown in \cite[Proposition 3.5]{BCZ}. Thus $\TT$ is widely generated.

  (3)
  By the second proof of (1), we have maps $\ov{\CJR} \colon (\tors\AA)_0 \to \sbrick\AA$ and $\TTT \colon \sbrick\AA \to (\tors\AA)_0$ which are mutually inverse to each other.
\end{proof}
By considering the opposite category $\AA^{\op}$ and the lattice anti-isomorphism $^\perp(-) \colon \torf\AA \to \tors\AA$, one obtains the following dual result. We omit the definition of $\ov{\CMR} \colon (\tors\AA)^0 \to \sbrick\AA$, which is dual to $\ov{\CJR} \colon (\tors\AA)_0 \to \sbrick\AA$.
\begin{corollary}\label{cor:cmr-sbrick}
  Let $\AA$ be an abelian length category.
  Then we have the following mutually inverse bijections between $\sbrick\AA$ and $(\tors\AA)^0$.
    \[
      \begin{tikzcd}[column sep = large]
        \sbrick \AA \rar["^\perp (-)", shift left]
        & (\tors\AA)_0 \lar["\ov{\CMR}", shift left]
      \end{tikzcd}
    \]
\end{corollary}

Combining Theorem \ref{thm:cjr-sbrick} and its dual Corollary \ref{cor:cmr-sbrick}, we obtain the following representation-theoretic interpretation of the extended kappa map. This also extends the bijections $\jirr (\tors\AA) \iso \brick\AA \iso \mirr(\tors\AA)$ given in Theorem \ref{thm:btz}.

\begin{theorem}[{c. f. \cite[Corollary 4.4.3]{BTZ}}]\label{thm:ext-kappa-bij}
  Let $\AA$ be an abelian length category. Then the extended kappa map gives a bijection $\kappa \colon (\tors\AA)_0 \isoto (\tors\AA)^0$. Moreover, we have the following commutative diagram consisting of bijections.
  \[
    \begin{tikzcd}
      (\tors\AA)_0 \rar["\ov{\CJR}"', shift right] \ar[rr, bend left, "\exk"]
      & \sbrick\AA \rar["^\perp(-)", shift left] \lar["\TTT"', shift right]
      \dar[shift left, "\Filt"]
      & (\tors\AA)^{0} \lar["\ov{\CMR}", shift left] \\
      & \wide\AA \uar[shift left, "\simp"] \ar[ul, bend left, "\TTT"]
      \ar[ur, bend right, "^\perp(-)"']
    \end{tikzcd}
  \]
  In particular, we have $\exk(\TTT(\SS)) = {}^\perp \SS$ for $\SS \in \sbrick\AA$ and $\exk(\TTT(\WW)) = {}^\perp \WW$ for $\WW \in \wide\AA$.
\end{theorem}

\begin{proof}
  The fact that the above diagram except $\exk$ is a commutative diagram consisting of bijections is shown in Theorem \ref{thm:cjr-sbrick} and Corollary \ref{cor:cmr-sbrick}. Thus it suffices to check $\exk(\TTT(\SS)) = {}^\perp \SS$ for a semibrick $\SS$. Since the expression $\TTT(\SS) = \bigvee \{ \TTT(B) \mid B \in \SS \}$ is a canonical join representation by Theorem \ref{thm:cjr-sbrick}, we have $\exk(\TTT(\SS)) = \bigwedge \{ \kappa(\TTT(B)) \mid B \in \SS \} = \bigwedge \{ {}^\perp B \mid B \in \SS \} = \bigcap_{B \in \SS} {}^\perp B = {}^\perp \SS$,
  where the second equality follows from Theorem \ref{thm:btz}.
\end{proof}
In particular, Question \ref{q:1} is true for $L = \tors\AA$.

\begin{remark}
  By considering torsion-free classes and the map $\WR \colon \torf\AA \to \wide\AA$ which is dual to $\WL$, or equivalently, by considering the opposite category $\AA^{\op}$ and $\WL \colon \tors (\AA^{\op}) \to \wide (\AA^{\op})$, the diagram in Theorem \ref{thm:ext-kappa-bij} can be completed into the following larger diagram consisting of bijections.
  From this diagram, we get the impression that the extended kappa map connects two kinds of dualities: the duality between $\AA$ and its opposite category $\AA^{\op}$, and the duality between $\tors\AA$ and $\torf\AA$ induced by perpendicular categories.
  \[
    \begin{tikzcd}
      & &
      (\tors\AA)^0 \dar[shift left, "(-)^\perp"] \ar[dl, "\ov{\CMR}"']
      \\
      (\tors\AA)_0 \rar["\ov{\CJR}"', shift right]
      \ar[urr, "\exk", bend left=15] \dar[equal]
      & \sbrick\AA \rar["\FFF", shift left] \lar["\TTT"', shift right]
      \dar[shift left, "\Filt"]
      & (\torf\AA)_0 \lar["\ov{\CJR}", shift left] \dar[equal]
      \uar[shift left, "^\perp(-)"]
      \\
      (\tors\AA)_0 \rar[shift left, "\WL"]
      & \wide\AA \uar[shift left, "\simp"] \lar["\TTT", shift left]
      \rar[shift right, "\FFF"']
      & (\torf\AA)_0 \lar[shift right, "\WR"']
    \end{tikzcd}
  \]
\end{remark}

\subsection{The kappa order and wide subcategories}
We have established a bijection $\TTT \colon \wide\AA \isoto (\tors\AA)_0$ in Theorem \ref{thm:cjr-sbrick}.
In this and next subsections, we recover the poset $\wide\AA$ using this bijection. More precisely, we define two partial orders on $L_0$ for a completely semidistributive lattice $L$: the \emph{kappa order $\leq_\kappa$} and the \emph{core label order $\leq_\clo$}, and show that $(\tors\AA)_0$ with these poset structures are isomorphic to $\wide\AA$ as posets.

First, we introduce the kappa order, which is defined using the extended kappa map.
\begin{definition}
  Let $L$ be a completely semidistributive lattice. Define a binary relation $\leq_\kappa$ on the set $L_0$ of elements with canonical join representations as follows: $a \leq_\kappa b$ if both $a \leq b$ and $\exk(a) \geq \exk(b)$ hold. This relation clearly gives a poset structure on $L_0$, which we call the \emph{kappa order}, and we denote by $L_\kappa$ the poset $(L_0, \leq_\kappa)$.
\end{definition}

The following is the main theorem of this subsection.
\begin{theorem}\label{thm:kappa-wide}
  Let $\AA$ be an abelian length category. Then the map $\TTT \colon \wide\AA \to \tors\AA$ induces a poset isomorphism $\wide\AA \isoto (\tors\AA)_\kappa$.
\end{theorem}
To prove this, we need the following general observation of subcategories. The author would like to thank Osamu Iyama for sharing the proof of this.
\begin{lemma}\label{lem:ie-closed}
  Let $\AA$ be an abelian length category and $\CC$ a subcategory of $\AA$ which is closed under images and extensions. Then the equality $\CC = \TTT(\CC) \cap \FFF(\CC)$ holds.
\end{lemma}
\begin{proof}
  Clearly $\CC$ is contained in $\TTT(\CC) \cap \FFF(\CC)$. Thus we only prove $\TTT(\CC) \cap \FFF(\CC) \subseteq \CC$.
  To this aim, we will need the following well-known description: $\TTT(\CC) = \Filt(\Fac \CC)$ and $\FFF(\CC) = \Filt(\Sub\CC)$ hold, where $\Fac \CC$ (resp. $\Sub\CC$) consists of $M$ such that there is a surjection $C \defl M$ (resp. an injection $M \hookrightarrow C$) with $C \in \CC$. See e.g. \cite[Lemma 3.1]{MS} for a proof.
  We divide the proof into three steps.

  \un{(Step 1): $\Fac\CC\cap\Sub\CC \subseteq \CC$}. Let $X$ be an object in $\Fac \CC \cap \Sub \CC$.
  Then there exist a surjection $C_1 \defl X$ with $C_1 \in \CC$ and an injection $X \hookrightarrow C_2$ with $C_2 \in \CC$. Thus $X$ is the image of the composition $\varphi \colon C_1 \defl X \hookrightarrow C_2$. Since $\CC$ is closed under images, we have $X = \im \varphi \in \CC$.

  \un{(Step 2): $\TTT(\CC) \cap \Sub \CC \subseteq \CC$}. Let $X$ be in $\TTT(\CC) \cap \Sub \CC$.
  Recall that $\TTT(\CC) = \Filt(\Fac\CC)$.
  We will show $X \in \CC$ by induction on the $(\Fac\CC)$-filtration length $n$ of $X$. If $n=1$, then this follows from (Step 1).
  Suppose $n>1$. There is a short exact sequence
    \[
    \begin{tikzcd}
      0 \rar & Y \rar & X \rar & Z \rar & 0,
    \end{tikzcd}
    \]
  where $Z$ is in $\Fac\CC$ and the $(\Fac\CC)$-filtration length of $Y$ is smaller than $n$. Since $X$ is in $\Sub\CC$, so is $Y$. By the induction hypothesis, we have $Y\in\CC$.
  Since $Z$ is in $\Fac\CC$, there is a surjection $C \defl Z$ with $C \in\CC$. Then we obtain the following pullback diagram.
    \[
    \begin{tikzcd}
      0 \rar & Y \rar \dar[equal]
      & E \dar[twoheadrightarrow] \rar \ar[rd, phantom, "{\rm p.b.}"]
      & C \rar \dar[twoheadrightarrow] & 0 \\
      0 \rar & Y \rar & X \rar & Z \rar & 0
    \end{tikzcd}
    \]
  Since $\CC$ is closed under extensions, we have $E \in \CC$. Then $X$ is in $\Fac\CC$, thus we obtain $X \in \Fac\CC \cap \Sub \CC \subseteq \CC$ by (Step 1).

  \un{(Step 3): $\TTT(\CC) \cap \FFF(\CC) \subseteq \CC$}. Let $X$ be in $\TTT(\CC) \cap \FFF(\CC)$.
  Recall that $\FFF(\CC) = \Filt(\Sub\CC)$ holds.
  We show $X \in \CC$ by the induction on the $(\Sub\CC)$-filtration length $n$ of $X$. If $n=1$, then this follows from (Step 2).
  Suppose $n>1$. There is a short exact sequence
    \[
    \begin{tikzcd}
      0 \rar & Y \rar & X \rar & Z \rar & 0,
    \end{tikzcd}
    \]
  where $Y$ is in $\Sub\CC$ and the $(\Sub\CC)$-filtration length of $Z$ is smaller than $n$. Since $X$ is in $\TTT(\CC)$, so is $Z$. By the induction hypothesis, we have $Z\in\CC$. Since $Y$ is in $\Sub\CC$, there is an injection $Y \hookrightarrow C$ with $C\in \CC$. Then we can take the following pushout diagram.
    \[
    \begin{tikzcd}
      0 \rar
      & Y \rar \dar[hookrightarrow] \ar[rd, phantom, "{\rm p.o.}"]
      & X \dar[hookrightarrow] \rar & Z \rar \dar[equal] & 0 \\
      0 \rar & C \rar & E \rar & Z \rar & 0
    \end{tikzcd}
    \]
  Since $\CC$ is closed under extensions, $E\in\CC$ holds, hence $X$ is in $\Sub\CC$. By (Step 2), we have $X\in\CC$.
\end{proof}

Now we are ready to prove Theorem \ref{thm:kappa-wide}.
\begin{proof}[Proof of Theorem \ref{thm:kappa-wide}.]
  By Theorem \ref{thm:cjr-sbrick}, we have a bijection $\TTT \colon \wide\AA \isoto (\tors\AA)_0$. Therefore, in order to prove that this map is a poset isomorphism, it suffices to show that $\WW_1 \subseteq \WW_2$ holds if and only if $\TTT(\WW_1) \leq_\kappa \TTT(\WW_2)$ holds for $\WW_1, \WW_2 \in \wide\AA$.
  This can be proved as follows:
  \begin{align*}
    \TTT(\WW_1) \leq_\kappa \TTT(\WW_2)
    & \Longleftrightarrow \TTT(\WW_1) \leq \TTT(\WW_2) \text{ and }
    \exk(\TTT(\WW_1)) \geq \exk(\TTT(\WW_2))
    \text{ in $\tors\AA$} \\
    & \Longleftrightarrow \TTT(\WW_1) \subseteq \TTT(\WW_2) \text{ and }
    {}^\perp \WW_1 \supseteq {}^\perp \WW_2 \\
    & \Longleftrightarrow \TTT(\WW_1) \subseteq \TTT(\WW_2) \text{ and }
    ({}^\perp\WW_1)^\perp \subseteq ({}^\perp\WW_2)^\perp \\
    & \Longleftrightarrow \TTT(\WW_1) \subseteq \TTT(\WW_2) \text{ and }
    \FFF(\WW_1) \subseteq \FFF(\WW_2) \\
    & \Longleftrightarrow \WW_1 \subseteq \WW_2
  \end{align*}
  Here the second equivalence follows from Theorem \ref{thm:ext-kappa-bij}, the third from the poset anti-isomorphism $(-)^\perp \colon \tors\AA \to \torf\AA$, and the last as follows: the implication $\Leftarrow$ is clear, and the converse $\Rightarrow$ follows from Lemma \ref{lem:ie-closed} since wide subcategories are closed under images and extensions.
\end{proof}

\subsection{The core label order and wide subcategories}
In this subsection, we give another poset structure on $L_0$ for a completely semidistributive lattice $L$.

First let us mention the terminology and the background of the core label order. Reading studied the poset of regions associated to a hyperplane arrangement and
introduced the \emph{shard intersection order} in \cite{reading-shard}, which is another poset (actually lattice) structure on the poset of regions. The typical example is the shard intersection order on a finite Coxeter group $W$, where the poset of regions is precisely the weak order on $W$.

Then the shard intersection order was generalized to another poset structure on an arbitrary finite \emph{congruence-uniform} lattice in \cite[Section 9-7.4]{reading-region}. Here we omit the definition of congruence-uniform lattice, but we only note that congruence-uniform lattices are special cases of semidistributive lattices. Then the term \emph{core label order} was introduced by M\"uhle in \cite{muhle} to distinguish this lattice-theoretically defined partial order and Reading's geometrically defined lattice structure.
This core label order on (particular) congruence-uniform lattices was studied in several authors, e.g. in \cite{cjc-biclosed, gm-flip, muhle}.

Since the definition of the core label order on finite congruence-uniform lattices has a natural generalization to (possibly infinite) completely semidistributive lattices, we only state it.
Recall that $L_0$ is the set of elements of $L$ with canonical join representations, and also note that $L= L_0$ holds for a finite semidistributive lattice by Proposition \ref{prop:cjr-char}.
\begin{definition}\label{def:clo}
  Let $L$ be a completely semidistributive lattice.
  \begin{enumerate}
    \item Let $x \in L_0$. Define $x_\downarrow$ as follows:
    \[
      x_\downarrow = x \wedge \bigwedge \{ x' \in L \mid x' \covd x \}.
    \]
    \item For $x, y \in L_0$, we write $x \leq_\clo y$ if $\jlabel [x_\downarrow, x] \subseteq \jlabel [y_\downarrow, y]$. We denote by $L_\clo$ the poset $(L_0, \leq_\clo)$.
  \end{enumerate}
\end{definition}
We also recall that $\jlabel$ can be regarded as considering the set of join-irreducible labels appearing in each interval by Theorem \ref{thm:jlabel-label}.

The fact that $\leq_\clo$ is a poset structure on $L_0$, namely, the fact that $x \leq_\clo y \leq_\clo x$ implies $x = y$, follows from the equality $x = \bigvee \jlabel [x_\downarrow, x]$, which can be proved as follows.
For $x \in L_0$, we have $x = \bigvee \CJR(x) = \bigvee \jlabel_\downarrow x$ by Lemma \ref{lem:cjr-label}. On the other hand, we clearly have $\jlabel_\downarrow x \subseteq \jlabel [x_\downarrow, x]$, hence we obtain $x = \bigvee \jlabel_\downarrow x \leq \bigvee \jlabel [x_\downarrow, x] \leq x$, which shows the assertion.

Now we can prove the following result on the core label order.
\begin{theorem}\label{thm:clo-wide}
  Let $\AA$ be an abelian length category. Then the map $\TTT \colon \wide\AA \to \tors\AA$ induces a poset isomorphism $\wide\AA \isoto (\tors\AA)_\clo$.
\end{theorem}
\begin{proof}
  Since we have a bijection $\TTT \colon \wide\AA \isoto (\tors\AA)_0$ by Theorem \ref{thm:ext-kappa-bij}, we only have to show that $\TTT(\WW_1) \subseteq \TTT(\WW_2)$ for $\WW_1,\WW_2 \in \wide\AA$ if and only if $\TTT(\WW_1) \leq_\clo \TTT(\WW_2)$.

  By Theorem \ref{thm:itv-char}(1), we have that $\WW_i$ is the heart of $[\TTT(\WW_i) \wedge {}^\perp\WW_i, \TTT(\WW_i)]$ for $i=1,2$. On the other hand, the equality $\TTT(\WW_i)_\downarrow = \TTT(\WW_i) \wedge {}^\perp \WW_i$ is known, see e.g. \cite[Proposition 3.3]{ES}.
  Combining these facts with Lemma \ref{lem:heart-brick}, we can prove the assertion as follows.
  \begin{align*}
    \WW_1 \subseteq \WW_2
    & \Longleftrightarrow \brick \WW_1 \subseteq \brick \WW_2 \\
    & \Longleftrightarrow \jlabel [\TTT(\WW_1)_\downarrow, \TTT(\WW_1)]
    \subseteq  \jlabel [\TTT(\WW_2)_\downarrow, \TTT(\WW_2)] \\
    & \Longleftrightarrow \TTT(\WW_1) \leq_\clo \TTT(\WW_2)
  \end{align*}
  Here the first equivalence follows from Lemma \ref{lem:heart-brick}, and the second from the commutative diagram (\ref{eq:diagram}) in the proof of Theorem \ref{thm:main}.
\end{proof}

Combining Theorems \ref{thm:kappa-wide} and \ref{thm:clo-wide}, we obtain the following two descriptions for $\wide \AA$:
\begin{corollary}\label{cor:two-coincide-wide}
  Let $\AA$ be an abelian length category. Then $(\tors\AA)_\kappa$ and $(\tors\AA)_\clo$ coincide, and the map $\TTT \colon \wide\AA \to \tors\AA$ induces a poset isomorphism between $\wide\AA$ and these posets. 
\end{corollary}

\begin{remark}
  Let us mention several results of Garver and McConville in \cite{gm-flip,gm-tiling}. They describe the posets $\tors\Lambda_T$ and $\wide\Lambda_T$ combinatorially for a \emph{tiling algebra $\Lambda_T$}, which is a particular representation-finite algebra associated with a tree $T$ embedded in a disk.
  In \cite{gm-tiling}, they proved that $\tors\Lambda_T$ and $\wide\Lambda_T$ are isomorphic to the \emph{oriented flip graph $\overrightarrow{FG}(T)$} and the lattice of \emph{noncrossing tree partitions $NCP(T)$} as lattices.
  On the other hand, in \cite{gm-flip}, they proved that $NCP(T)$ is isomorphic to $(\overrightarrow{FG}(T), \leq_\clo)$.
  In particular, they proved that $\wide\Lambda_T$ is isomorphic to $(\tors\Lambda_T, \leq_\clo)$. Thus our result can be regarded as a generalization of their result to any abelian length category, as expected in the introduction of \cite{gm-tiling}.
\end{remark}

In general, $L_\kappa$ and $L_\clo$ do not coincide even for a finite congruence-uniform lattice $L$, as the following example shows.
\begin{example}[taken from {\cite[Figure 7]{muhle}}]\label{ex:not-same}
  Let $L$ be the following lattice, where we show the Hasse quiver and its join-irreducible labeling (the label $n$ corresponds to $j_n$).
  \[
    \begin{tikzcd}
      & 1 \ar[ld, "3" description] \dar["2" description] \ar[rd, "1" description] \\
      x \ar[dd, "2" description] \ar[rd, "1" description]
      & y \ar[ldd, "3" description] \ar[rdd, "1" description]
      & z \ar[dd, "2" description] \ar[ld, "3" description] \\
      & j_4 \dar["4" description] \\
      j_1 \ar[rd, "1" description]
      & j_2 \dar["2" description]
      & j_3 \ar[ld, "3" description] \\
      & 0
    \end{tikzcd}
  \]
  This $L$ is completely semidistributive, and moreover, it is congruence-uniform. Since $L$ is finite, we have $L = L_0$. However, $\leq_\clo$ and $\leq_\kappa$ does not coincide, as Figure \ref{fig:not-same} shows. For example, we have $j_4 \leq_\clo x$ since $\jlabel [(j_4)_\downarrow, j_4] = \jlabel[j_2,j_4] = \{j_4\}$ and $\jlabel[x_\downarrow, x] = \jlabel[0,x] = \{j_1,j_2,j_4\}$.
  On the other hand, we have $j_4 \leq x$ but $\exk(j_4) = j_2 \not\geq \exk(x) = j_3$, hence $j_4 \not\leq_\kappa x$.

  \begin{figure}[htp]
    \centering
    \caption{Two poset structures on $L$}
    \label{fig:not-same}
    \begin{minipage}{0.45\textwidth}
      \centering
      \begin{tikzcd}[column sep = small]
        & 1 \ar[dl] \dar \ar[dr] \ar[ddrr, bend left] \\
        x \dar \ar[dr]  &
        y \ar[dl] \ar[dr] &
        z \ar[dl] \dar \\
        j_1 \ar[dr] & j_2 \dar & j_3 \ar[dl] & j_4 \ar[dll] \\
        & 0
      \end{tikzcd}
      \subcaption{The kappa order $(L,\leq_\kappa)$}
    \end{minipage}
    \begin{minipage}{0.45\textwidth}
      \centering
      \begin{tikzcd}[column sep = small]
        & 1 \ar[dl] \dar \ar[dr] \\
        x \dar \ar[dr] \ar[drrr] &
        y \ar[dl] \ar[dr] &
        z \ar[dl] \dar \ar[dr] \\
        j_1 \ar[dr] & j_2 \dar & j_3 \ar[dl] & j_4 \ar[dll] \\
        & 0
      \end{tikzcd}
      \subcaption{The core label order $(L,\leq_\clo)$}
    \end{minipage}
  \end{figure}
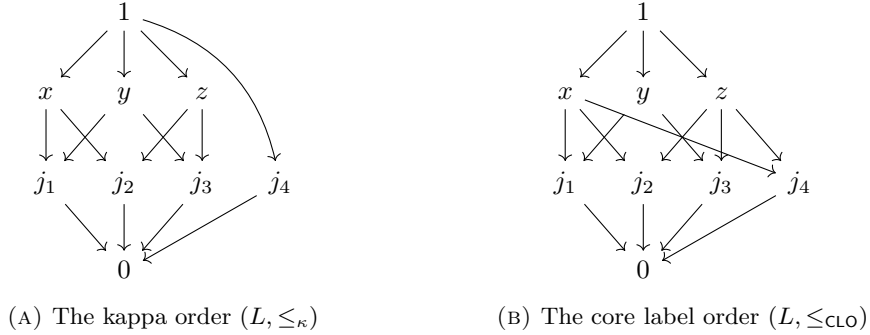
\end{example}

Therefore, we have the following natural question.
\begin{question}
  Let $L$ be a completely semidistributive lattice. When do the kappa order and the core label order on $L_0$ coincide?
\end{question}
We have the following sufficient condition, which we will use later.
\begin{proposition}\label{prop:coincide-condition}
  Let $L$ be a completely semidistributive lattice. Suppose that the following equality holds for each $x \in L_0$:
  \[
    \jlabel [x_\downarrow, x]
    = \{ j \in \jirr L \mid \text{$j \leq x$ and $\kappa(j) \geq \exk(x)$}\}.
  \]
  Then the two partial orders $\leq_\kappa$ and $\leq_\clo$ on $L_0$ coincide.
\end{proposition}
\begin{proof}
  Let $x, y \in L_0$. Suppose that $x \leq_\kappa$ holds, that is, $x \leq y$ and $\exk(x) \geq \exk(y)$. Then for each $j \in \leq_\kappa$, if $j \leq x$ and $\kappa(j) \geq \exk(x)$, then we have $j \leq x \leq y$ and $\kappa(j) \geq \exk(x) \geq \exk(y)$. This, combined with the assumed equality, implies $\jlabel[x_\downarrow, x] \subseteq \jlabel[y_\downarrow, y]$ that is, $x \leq_\clo y$.

  Conversely, suppose that $x \leq_\clo y$ holds.
  We claim that $z = \bigvee \jlabel[z_\downarrow, z]$ and $\exk(z) = \bigwedge \{\kappa(j) \mid j \in \jlabel[z_\downarrow, z]\}$ hold for each $z \in L_0$. The first equality is shown in the observation below Definition \ref{def:clo}. On the other hand, each $j \in \jlabel[z_\downarrow,z]$ satisfies $\kappa(j) \geq \exk(z)$ by the assumption, and $\CJR(z) = \jlabel_\downarrow z \subseteq \jlabel[z_\downarrow, z]$ by Lemma \ref{lem:cjr-label}. Therefore, we have
  \[
    \exk(z) = \bigwedge \{ \kappa(j) \mid j \in \CJR(z)\}
    \geq \bigwedge \{ \kappa(j) \mid j \in \jlabel[z_\downarrow, z]\}
    \geq \exk(z),
  \]
  which implies the claim.
  Now suppose that $x \leq_\clo y$ holds for $x,y \in L_0$, that is, $\jlabel[x_\downarrow, x] \subseteq \jlabel[y_\downarrow, y]$ holds. By taking the join and the meet, the claim immediately implies $x \leq y$ and $\exk(x) \geq \exk(y)$.
\end{proof}
We note that we can show that $L = \tors\AA$ for an abelian length category $\AA$ satisfies this condition by using Lemma \ref{lem:ie-closed}.

\begin{example}\label{ex:2}
  Let $k$ be a field and $Q$ the following quiver.
  \[
    \begin{tikzcd}[sep = tiny]
      1 \ar[rr, "a"] & & 2 \ar[dl, "b"] \\
      & 3 \ar[ul, "c"]
    \end{tikzcd}
  \]
  Consider the algebra $\Lambda := kQ/\la ab, bc, ca \ra$, and put $L:= \tors\Lambda$.
  In Figure \ref{fig:ex2}, we show the Hasse quiver of $L$ together with the join-irreducible labeling and the Hasse diagram of $L_\kappa = L_\clo$.
  \begin{figure}[htp]
    \caption{$\tors\Lambda$ and $\wide\Lambda$}
    \label{fig:ex2}
    \begin{minipage}{0.45\textwidth}
      \centering
      \begin{tikzpicture}[scale = 0.9]
        \begin{scope}
          \node[circle] (a) at (0,0) {$0$};
          \node (b) at (0,2.5) {$1$};
          \node (c) at (0,5) {$\ov{5}$};
          \node (d) at (1,2.5) {$5$};
          \node (e) at (1,5) {$\ov{3}$};
          \node (f) at (1.7, 1.7) {$2$};
          \node (g) at (1.7,3.3) {$4$};
          \node (h) at (3.3,1.7) {$\ov{4}$};
          \node (i) at (3.3,3.3) {$\ov{1}$};
          \node (j) at (4,0) {$3$};
          \node (k) at (4,2.5) {$\ov{6}$};
          \node (l) at (5,0) {$6$};
          \node (m) at (5,2.5) {$\ov{2}$};
          \node (n) at (5,5) {$\ov{0}$};
        \end{scope}
    
        \begin{scope}[every node/.style={blabel}]
          \draw[->] (k) -> node {$3$} (d);
          \draw[->] (c) to node {$2$} (b);
          \draw[->] (b) -> node {$1$} (a);
          \draw[->] (e) -> node[near end] {$2$} (d);
          \draw[->] (g) -> node[near end] {$4$} (f);
          \draw[->] (i) -> node[near start] {$4$} (h);
          \draw[->] (k) -> node[near start] {$1$} (j);
          \draw[->] (m) -> node {$1$} (l);
          \draw[->] (n) -> node {$2$} (m);
          \draw[->] (l) -> node {$6$} (j);
          \draw[->] (j) -> node {$3$} (a);
          \draw[->] (m) -> node {$6$} (k);
          \draw[->] (d) -> node {$5$} (b);
          \draw[->] (n) -> node {$3$} (e);
          \draw[->] (e) -> node {$5$} (c);
          \draw[->] (h) -> node {$3$} (f);
          \draw[->] (i) -> node {$3$} (g);
          \draw[->] (n) -> node {$1$} (i);
          \draw[->] (f) -> node {$2$} (a);
          \draw[->] (c) -> node[near start] {$1$} (g);
          \draw[->] (h) -> node[near end] {$2$} (l);
        \end{scope}
    
      \end{tikzpicture}
      \subcaption{The Hasse quiver of $\tors\Lambda$}
    \end{minipage}
    \begin{minipage}{0.45\textwidth}
      \centering
      \begin{tikzpicture}
        \node (00) at (2.5, 3) {$\ov{0}$};
        \node (11) at (5, 2) {$\ov{1}$};
        \node (22) at (4, 2) {$\ov{2}$};
        \node (33) at (3, 2) {$\ov{3}$};
        \node (44) at (2, 2) {$\ov{4}$};
        \node (55) at (1, 2) {$\ov{5}$};
        \node (66) at (0, 2) {$\ov{6}$};
        \node (1) at (0, 1) {$1$};
        \node (2) at (1, 1) {$2$};
        \node (3) at (2, 1) {$3$};
        \node (4) at (3, 1) {$4$};
        \node (5) at (4, 1) {$5$};
        \node (6) at (5, 1) {$6$};
        \node (0) at (2.5, 0) {$0$};
      
        \draw (00) edge (11) edge (22) edge (33) edge (44) edge (55) edge (66);
        \draw (0) edge (1) edge (2) edge (3) edge (4) edge (5) edge (6);
        \draw (66) edge (1) edge (3) edge (5);
        \draw (55) edge (1) edge (2) edge (4);
        \draw (44) edge (2) edge (3) edge (6);
        \draw (33) edge (2) edge (5);
        \draw (22) edge (1) edge (6);
        \draw (11) edge (3) edge (4);
      \end{tikzpicture}
      \subcaption{The Hasse diagram of $L_\kappa = L_\clo \cong \wide\Lambda$}
    \end{minipage}
  \end{figure}
  There are six join-irreducible elements $1,\dots, 6$ and $\kappa(i) = \ov{i}$ for each $i$.
  The orbit of $\exk$ is given by $1 \mapsto \ov{1} \mapsto 2 \mapsto \ov{2} \mapsto 3 \mapsto \ov{3} \mapsto 1$, $4 \mapsto \ov{4} \mapsto 5 \mapsto \ov{5} \mapsto 6 \mapsto \ov{6} \mapsto 4$, and $0 \mapsto \ov{0} \mapsto 0$.
  For example, we have $4 \leq_\clo \ov{5}$ since $\jlabel[4_\downarrow, 4] = \jlabel[2,4] =  \{4\}$ and $\jlabel [\ov{5}_\downarrow, \ov{5}] = \jlabel [0,\ov{5}] = \{1,2,4\}$, and we also have $4 \leq_\kappa \ov{5}$ since $4 \leq \ov{5}$ and $\exk(4) = \ov{4} \geq 6 = \exk(\ov{5})$.
  \begin{figure}[htp]
    \label{fig:ex2-wide}
  \end{figure}
\end{example}

\subsection{Combinatorial consequences}\label{sec:combi}
In this subsection, we consider some consequences of our results for particular classes of algebras.
In what follows, we fix a field $k$, and let $Q$ be a Dynkin quiver and $W$ its Coxeter group.

First, consider the path algebra $\Lambda:= kQ$. In \cite{IT}, a combinatorial description of the posets $\tors\Lambda$ and $\wide\Lambda$ are given as follows:
$\tors\Lambda$ is isomorphic the \emph{Cambrian lattice $\C_Q$} (\cite[Theorem 4.3]{IT}), and $\wide\Lambda$ is isomorphic to the \emph{non-crossing partition lattice $\NC(W)$} (\cite[Section 3]{IT}, see also \cite[Theorem 3.7.4.4]{ringel-catalan}). Moreover, Reading \cite[Theorem 10-6.34]{reading-region} showed that $(\C_Q, \leq_\clo)$ is isomorphic to $\NC(W)$.
Then Corollary \ref{cor:two-coincide-wide} implies the following result.
\begin{corollary}
  Let $Q$ be a Dynkin quiver, $W$ the Coxeter group of $Q$, and $\C_Q$ the Cambrian lattice. Then we have $(\C_Q, \leq_\kappa) = (\C_Q, \leq_\clo)$, and this poset is isomorphic to $\wide kQ$ and also to the lattice of non-crossing partitions $\NC(W)$.
\end{corollary}
The relation between our results and others is summarized as the following commutative diagram consisting of poset isomorphisms:
\[
  \begin{tikzcd}
    \wide kQ \rar[<->, "\text{\cite{IT}}"] \dar[<->, "(*)"']
    & \NC(W) \dar[<->, "\text{\cite{reading-region}}"] \\
    (\C_Q, \leq_\kappa) \rar[equal, "(*)"'] & (\C_Q, \leq_\clo)
  \end{tikzcd}
\]
Here $(*)$ are new results which follow from our study. For example, assuming Ingalls--Thomas's result, the above corollary provides a new proof of the fact shown in \cite{reading-region} that $\NC(W)$ is isomorphic to the core label order of $\C_Q$.

Next, we consider the preprojective algebra $\Pi_Q$, whose definition we omit. Mizuno proved in \cite[Theorem 2.30]{mizuno} that $\tors\Pi_Q$ is isomorphic to $(W,\leq)$, where $\leq$ is the right weak order. On the other hand, it is implicitly shown in \cite{thomas-shard} that $\wide \Pi_Q$ is isomorphic to $(W,\leq_\clo)$, that is, Reading's original shard intersection order on $W$.
In this case, Corollary \ref{cor:two-coincide-wide} implies the following consequence.
\begin{corollary}
  Let $Q$ be a Dynkin quiver, $W$ its Weyl group (with the right weak order), and $\Pi_Q$ its preprojective algebra. Then we have $(W,\leq_\kappa) = (W,\leq_\clo)$, and this poset is isomorphic to $\wide\Pi_Q$.
\end{corollary}
This provides a new simple description of the shard intersection order on $W$. Actually, we can extend this result to non-simply-laced case by using Gei\ss--Leclerc--Schr\"oer's generalized preprojective algebra \cite{GLS}. For a symmetrizable generalized Cartan matrix $C$ with a symmetrizer $D$, they defined an algebra $\Pi(C, D)$, which is finite-dimensional if $C$ is of Dynkin type.
Fu--Geng \cite{FG} extends Mizuno's result to this setting: he showed that $\tors\Pi(C,D)$ is isomorphic to the Weyl group $W$ of the Kac-Moody Lie algebra associated with $C$ if $C$ is of Dynkin type.
We omit the statement about generalized preprojective algebras. Instead, we give some applications to the shard intersection order on any finite Coxeter group.
\begin{proposition}\label{prop:coxeter-shard}
  Let $W$ be a finite Coxeter group together with the right weak order. Then the kappa order and the core label order (= the shard intersection order) on $W$ coincide.
\end{proposition}
\begin{proof}
  The standard argument on direct products shows that we may assume that $W$ is irreducible.
  It is well-known that a finite irreducible Coxeter group $W$ can be realized as the Weyl group of the Kac-Moody Lie algebra of a symmetrizable generalized Cartan matrix $C$ (i.e. a finite crystallographic reflection group) except certain cases: type $I_2(n)$, $H_3$ and $H_4$ (see e.g. \cite{hum}).
  If $W$ can be realized as the Weyl group, then $(W,\leq)$ is isomorphic to $\tors\Lambda$ for some finite-dimensional algebra by the above argument $\Lambda$.
  If $W$ is of type $I_2(n)$, then we can easily check the assertion (or one can construct an algebra $\Lambda$ such that $W$ is isomorphic to $\tors\Lambda$ as posets, see \cite[Proposition 6.1]{kase}).
  Finally, if $W$ is of type $H_3$ or $H_4$, then one can use SageMath \cite{sage} to verify the sufficient condition Proposition \ref{prop:coincide-condition}.
\end{proof}

\addtocontents{toc}{\SkipTocEntry}
\subsection*{Acknowledgement}
The author would like to thank Osamu Iyama for sharing him the proof of Lemma \ref{lem:ie-closed}. He would also like to thank Yuya Mizuno for helpful discussions.
This work is supported by JSPS KAKENHI Grant Number JP21J00299.

\end{document}